\documentclass[11pt,twoside]{article}
\usepackage{latexsym}
\usepackage{amssymb,amsbsy,amsmath,amsfonts,amssymb,amscd}
\usepackage{mathrsfs}
\usepackage{epsfig, graphicx}
\usepackage{graphics,epsfig}
\usepackage{color}
\usepackage{enumerate}
\usepackage{epstopdf}
\usepackage{titlesec}
\usepackage{subfigure}
\usepackage[utf8]{inputenc}
\usepackage[T1]{fontenc}
\usepackage{lmodern}
\usepackage{graphicx}
\usepackage{caption}
\usepackage{tikz}
\usetikzlibrary{arrows,positioning, arrows}

\usepackage{hyperref}

\usepackage{amsthm}
\usepackage{amsfonts}
\usepackage{graphicx}
\usepackage{caption}
\usepackage{amsmath}
\usepackage{amssymb}
\usepackage{amsmath}
\usepackage{bm}
\usepackage[linesnumbered,ruled]{algorithm2e}
\usepackage[noend]{algpseudocode}
\usepackage{cases}
\theoremstyle{plain}
\newtheorem{theorem}{Theorem}
\newtheorem{lemma}{Lemma}

\newtheorem{proposition}{Proposition}

\theoremstyle{definition} 

\newtheorem{remark}{Remark}

\textheight 9in \textwidth 6.5in \numberwithin{equation}{section}

\newcommand{\eps}{\varepsilon}

\newcommand{\half}{\frac{1}{2}}
\newcommand{\RR}{\mathbb{R}}
\newcommand{\rd}{\mathrm{d}}
\newcommand{\etaN}{\eta_{\mathcal{N}}}
\newcommand{\etaM}{\eta_{\mathcal{M}}}
\newcommand{\sumK}{\frac{1}{K} \sum_{k=0}^{K-1} }

\newcommand{\etasumK}{\frac{\eta}{K} \sum_{k=0}^{K-1} }
\newcommand{\kp}{{k+1}}
\newcommand{\jp}{{j+1}}
\newcommand{\np}{{n+1}}
\newcommand{\jm}{{j-1}}

\newcommand{\Du}{\Delta u}
\newcommand{\energy}{\mathcal{E}}
\newcommand{\dW}{\mathcal{W}_2}
\newcommand{\Wass}{\mathcal{W}}
\newcommand{\grad}{\nabla}
\newcommand{\KK}{\mathbb K}
\newcommand{\UU}{\mathbb U}
\newcommand{\Hnrhoi}{{\Delta_{\rho_n}^{-1}}}
\newcommand{\Hnrho}{{\Delta_{\rho_n}}}
\newcommand{\Hnui}{{\Delta_{M(u_n)}^{-1}}}
\newcommand{\Hnu}{{\Delta_{M(u_n)}}}
\newcommand{\Drhon}{\mathsf{D}_{\rho_n}}
\newcommand{\bfrho}{\boldsymbol{\rho}}
\newcommand{\bfy}{\boldsymbol{y}}
\newcommand{\bfb}{\boldsymbol{b}}
\newcommand{\bfh}{\boldsymbol{h}}
\newcommand{\llp}{{(l+1)}}
\newcommand{\ls}{{(l)}}
\newcommand{\matA}{\mathsf{A}}
\newcommand{\matP}{\mathsf{P}}
\newcommand{\average}[1]{\left\langle#1\right\rangle}
\newcommand{\variErho}{\frac{\delta \energy}{\delta \rho}}
\newcommand{\variEu}{\frac{\delta \energy}{\delta u}}

\newcommand{\my}[1]{{{\color{cyan}{MY: #1}}}}

\graphicspath{{../}}

\title{Hessian informed mirror descent \thanks{L.W. is partially supported by NSF grant DMS-1846854. M.Y is partially supported by NSF grant DMS-2012439.}}
\date{}
\author{Li Wang\thanks{School of Mathematics, University of Minnesota, Twin cities, MN 55455. (wang8818@umn.edu)}~  and Ming Yan\thanks{Department of Computational Mathematics, Science and Engineering and Department of Mathematics, Michigan State University, East Lansing, MI 48824. (myan@msu.edu)}}

\begin{document}
	\maketitle
	\begin{abstract}
		Inspired by the recent paper (L. Ying, Mirror descent algorithms for minimizing interacting free energy, Journal of Scientific Computing, 84 (2020), pp. 1–14), we explore the relationship between the mirror descent and the variable metric method. When the metric in the mirror decent is induced by a convex function, whose Hessian is close to the Hessian of the objective function,  this method enjoys both robustness from the mirror descent and superlinear convergence for Newton type methods. When applied to a linearly constrained minimization problem, we prove the global and local convergence, both in the continuous and discrete settings. As applications, we compute the Wasserstein gradient flows and Cahn-Hillard equation with degenerate mobility. When formulating these problems using a minimizing movement scheme with respect to a variable metric, our mirror descent algorithm offers a fast convergent speed for the underlining optimization problem while maintaining the total mass and bounds of the solution.  
	\end{abstract}
	


	
	\section{Introduction}
	We consider the following linearly constrained minimization problem
	\begin{equation} \label{eqn000}
	\min_{u:Au=b} f(u)\,,
	\end{equation}
	where $f:\Omega\rightarrow\mathbb R$ is a convex differentiable function and $A\in \mathbb R^{m\times n}$ with $m$ being a small nonnegative integer. When $m=0$, there is no constraint.  A typical form of $f$ reads as
	\begin{equation}
	f(u) = \sum_{i=1}^n (g_i(u_i) + u_i V_i) + \half \sum_{i=1,j=1}^{n,n} W_{i,j} u_i u_j\,,
	\end{equation}
	which arises in aggregation dynamics \cite{TBL06, CCY19}, kinetic description of granular gas \cite{Agueh16}, the mean field limit of neural networks \cite{MMN18}, among many others. In this paper, we assume that the problem~\eqref{eqn000} has a unique solution $u^*$. 
	
	When $A=\mathbf{1}^\top$ (the all one row vector), $b=1$, and $\Omega=\{u:u_i\geq 0\}$, i.e., the feasible set is the simplex 
	\begin{equation} \label{setK}
	\mathcal{U}=\left\{u:u_i \geq 0, \qquad \sum_{i=1}^n u_i = 1\right\}\,.
	\end{equation}
In this case, a strongly convex function $\Phi(u)$ is constructed to solve the problem \eqref{eqn000}, e.g., $\Phi(u)=\sum_{i=1}^ng_i(u_i)$ for the general case, and $\Phi(u)=\sum_{i=1}^ng_i(u_i)+{1\over2}W_{i,i}u_i^2$ if the matrix $[W_{i,j}]$ is positive semidefinite. 	Ying considered three different types of strongly convex functions $g_i$ in~\cite{Ying20}: Kullback-Leibler divergence, reverse Kullback-Leibler divergence, and Hellinger divergence. Then, the mirror descent has the following update formula
	\begin{align}
	\nabla \Phi (u^{k+1}) = \nabla \Phi (u^{k})  -  \etaM  (\nabla f(u^k)+A^\top c(u^k))\,,  \label{MD1}
	\end{align}	
 	where $\etaM$ is the stepsize and $c(u^k)\in\mathbb R^m$ is the unique vector to be determined such that $A u^{k+1}=b$. The nonnegative conditions $\{u_i\geq0\}_{i=1}^n$ are automatically satisfied because of the $\log$ terms in $\{g_i\}_{i=1}^n$, and $c(u^k)$ plays the role of the Lagrangian multiplier for the constraint $A u=b$ in the mirror descent update. For the special case when $[\nabla \Phi(u)]_i=\log(u_i)$ and $A=\mathbf{1}^\top$, the value $c(u^k)$ can be easily found by a normalization step. For other cases in~\cite{Ying20}, the value for $c(u^k)$ is efficiently found by iterative algorithms such as Newton and bisection.

	To put \eqref{MD1} in a more general framework, let $\Phi^*$ be the conjugate function of $\Phi$, which is defined as 
	$\Phi^*(v)= \max_{u}~ v^\top u - \Phi(u)$.
	Then we have $u=\nabla\Phi^*(\nabla \Phi (u))$.
	Therefore, the update of $u^{k+1}$ is 
	\begin{align*}
	u^{k+1}=&\nabla\Phi^*(\nabla\Phi(u^{k+1}))=\nabla \Phi^*(\nabla \Phi (u^{k})  -  \etaM  (\nabla f(u^k)+A^\top c(u^k))).
	\end{align*}
	Since $u=\nabla \Phi^*(\nabla \Phi(u))$, taking derivative with respect to $u$, we have $\mathbf{I}= \nabla^2\Phi(u)\nabla^2 \Phi^*(\nabla \Phi(u)).$ Then the above equation bares the following first order approximation: 
	\begin{align*}
	u^{k+1}\approx&\nabla \Phi^*(\nabla \Phi (u^{k})) - \etaM\nabla ^2\Phi^*( \nabla\Phi(u^k))(\nabla f(u^k)+A^\top c(u^k))\\
	=&u^{k} -  \etaM\nabla ^2\Phi(u^k)^{-1} (\nabla f(u^k)+A^\top c(u^k))\,.
	\end{align*}
	 It shows that mirror descent is a discretization of 
	\begin{align}
	& \dot{u} = - \nabla^2 \Phi(u)^{-1} (\nabla f(u)+A^\top c(u))\,.  \label{qNewton}
	\end{align}
	Since we can multiply $\Phi$ by a scalar and change the ordinary differential equation, we assume that $\Phi$ is 1-strongly convex with respect to a given norm $\|\cdot\|_w$ to simplify the following analysis.
	A more direct discretization of~\eqref{qNewton} is to apply the forward Euler scheme, namely, 
	\begin{align}
	u^{k+1} = u^k -  \etaN \nabla^2 \Phi (u^k)^{-1} (\nabla f(u^k)+A^\top c(u^k))\,, \label{qNewton1}
	\end{align}	
	which can be viewed as a first order variant of the mirror descent.
	Similarly, $\etaN$ is the stepsize and $c(u^k)$ is a vector to be determined such that $A u^{k+1}=b$. This method is equivalent to 
	\begin{align} \label{pNewton}
	u^{k+1} = \arg\min_{u:Au=b} f(u^k)+\langle\nabla f(u^k),u-u^k\rangle +{1\over2\etaN}\|u-u^k\|_{\nabla^2\Phi(u^k)}^2,
	\end{align}	
	and it is called variable metric because of the variable metric $\nabla^2\Phi(u^k)$ used in the quadratic term. 
	When $\Phi = f$, \eqref{pNewton} reduces to the proximal Newton method \cite{LSS14}.
	
	
	In view of the mirror descent method \eqref{MD1} and variable metric method \eqref{qNewton1}, they both are first order discretizations of the continuous flow \eqref{qNewton}. Despite vast literature on either method individually, there is little discussion on the relation between them. Indeed, for the mirror descent method, emphasize has been put on the treatment of constraints, especially the simplex constraint mentioned previously, which makes the choice of $\Phi(u) = \sum_{i=1}^nu_i\log u_i$ the most popular. On the other hand, in variable metric methods such as Newton type methods, $\Phi$ is chosen to incorporate the second order information of the objective function with the goal of improving the local convergence rate. The constraint, however, is often dealt with by a projection step. Inspired by the paper~\cite{Ying20}, we see that one can merge the advantages of both methods by constructing $\Phi$ that has both Hessian information and constraint guarantee. Consequently, by choosing the appropriate Bregman divergence in the mirror descent, we can prove the global convergence of the new method. This proof can easily lend itself to Newton type methods owing to their similarity. In return, following the superlinear convergence for Newton type methods, we can prove the same local convergence for the new method. 

	The contributions and organization of this paper are summarized as follows: 
	\begin{itemize}
		\item We establish the sublinear convergence of the gradient flow in~\eqref{qNewton} for a general $\Phi(u)$ in Section~\ref{sec:con} and extend it to linear convergence with an improved rate in the case of strong convexity;
		\item We prove both the global and local convergence of two discreterizations~\eqref{MD1} and~\eqref{qNewton1} in Section~\ref{sec:dis}.
		\item Applications in variable metric gradient flows are presented in Section 4 along with numerical experiments. 
	\end{itemize}
    Finally, the conclusion is drawn in Section 5.

	\section{Convergence of the gradient flow~(\ref{qNewton})}\label{sec:con}
	In this section, we consider the convergence of \eqref{qNewton}, 
	which guides the convergence analysis of~\eqref{MD1} and~\eqref{qNewton1} in the next section. With the proper choice of distance measure, in particular the Bregman divergence in our case, the global convergence can be established. 
	
	\begin{theorem}[Sublinear convergence] \label{thm:global-cont-mirror}
		Let $u(t)$ be the solution to \eqref{qNewton} with $u(0)=u_0$. Then we have 
		\begin{equation} \label{ineq1}
		f\left( \frac{1}{T} \int_0^T u(t) \rd t  \right) - f(u^*) \leq \frac{1}{T} D_\Phi(u^*,u_0)\,,
		\end{equation}
		where $D_\Phi$ is the Bregman divergence induced by $\Phi$:
		\[
		D_\Phi(u^*,u_0) = \Phi(u^*) - \Phi(u_0) - \nabla \Phi (u_0)^\top(u^*-u_0)\,.
		\]
	\end{theorem}
	\begin{proof}
		Consider the time derivative of $D_\Phi(u^*, u(t))$, we have
		\begin{align*}
		\frac{d}{dt} D_\Phi(u^*, u(t)) & = \frac{d}{dt} \left[   \Phi(u^*) - \Phi(u(t)) - \nabla \Phi(u(t))^\top (u^* - u(t))  \right]
		\\ & = -\nabla \Phi(u(t))^\top \frac{d u(t)}{dt} - \left[ \nabla^2 \Phi (u) \frac{du(t)}{dt}\right]^\top(u^*-u(t)) + \nabla \Phi(u(t))^\top \frac{d u(t)}{dt}
		\\ & = (\nabla f(u(t))+A^\top c(u(t)))^T(u^*-u(t))=\nabla f(u(t))^\top(u^*-u(t))
		\\ & \leq f(u^*) - f(u(t))\,,
		\end{align*}
		where the third equality uses \eqref{qNewton} and the inequality comes from the convexity of $f$. Integrating both sides from 0 to $T$, we obtain
		\begin{align*}
		\frac{1}{T} \left[ D_\Phi(u^*, u(T)) - D_\Phi(u^*, u_0) \right] \leq f(u^*) - {1\over T}\int_0^T f(u(t)) \rd t \,,
		\end{align*}
		which readily implies \eqref{ineq1} thanks again to the convexity of $f$ and $D_\Phi$ being nonnegative.
	\end{proof}
	
If we further assume the strong convexity of $f$, we can obtain the linear convergence. 
\begin{theorem}[Linear convergence] \label{thm2}
		Let $u(t)$ be the solution to \eqref{qNewton} with $u(0)=u_0$. Define two Bregman divergences induced by $\Phi$ and $f$ as $D_\Phi(t):=D_\Phi(u^*,u)=\Phi(u^*)-\Phi(u)-\langle\nabla\Phi(u),u^*-u\rangle$ and $D_f(t):=D_f(u^*,u)=f(u^*)-f(u)-\langle\nabla f(u),u^*-u\rangle$, respectively. Assume that $D_f(t)\geq \mu D_\Phi(t)$ for all $t$. Then we have 	
		
		\begin{equation*}
		D_\Phi(t)\leq D_\Phi(0)\exp^{-\mu t}\,,  \qquad \text{for all } ~ t \geq t_0\,.
		\end{equation*}
	\end{theorem}
	
	\begin{proof}
	Denote $G(u) = - [\nabla f(u) + A^\top c(u)]$, then \eqref{qNewton} writes as
	\begin{equation}\label{eqn0521}
	\frac{d}{dt}{\nabla \Phi(u)} = G(u)\,, \qquad \textrm{or} \quad \dot{u} =  \nabla^2\Phi(u)^{-1} G(u)\,.
	\end{equation}
	The global convergence result shows $G(u^*) = 0$.
	Then, we have 
	\begin{align*}
	   \dot{D_\Phi} & = -\langle G(u),u^*-u\rangle=-\langle\nabla f(u^*)+A^\top c(u^*) -\nabla f(u)-A^\top c(u),u^*-u\rangle \\
	   & = -\langle \nabla f(u^*)-\nabla f(u),u^*-u\rangle \qquad \textrm{ from }(Au=Au^*)\\
	   &\leq -[f(u^*)-f(u)-\langle \nabla f(u),u^*-u\rangle]\leq -\mu D_\Phi(t).
	   \end{align*}
Therefore, we have 
$D_\Phi(t)\leq D_\Phi(0)\exp^{-\mu t}$.
	\end{proof}


	\begin{remark}
	The scalar $\mu$ determines the linear convergence rate. If $\Phi(u) = \|u\|^2/2$, then $\mu$ is the strongly convex constant with respect to the standard norm, which can be very small in some applications. In such cases, if $\Phi$ is chosen according to the Hessian of $f$, then $\mu$ can be much larger than the strongly convex constant of $f$ with respect to the standard norm and results in a much faster convergence. 
	\end{remark}

	\section{Convergence at the discrete level}\label{sec:dis}
	This section is devoted to the convergence of the discrete schemes \eqref{MD1} and \eqref{qNewton1}. For global convergence, the proof follows a similar line of reasoning as in the continuous setting but with more involved calculations; whereas the local convergence is obtained via a two stage proof as in other Newton type methods. Throughout the section, we will use the following notation for simplicity
	\begin{equation} \label{diffu}
	\Delta u = u^\kp - u^k.
	\end{equation}

	\subsection{Global convergence} 
	We first establish the global convergence of \eqref{MD1}, which is slightly different from that in~\cite{beck2003mirror}. We still include it here for completeness. 
	\begin{theorem}[Global sublinear convergence for mirror descent
	\eqref{MD1}]\label{thm-dis-mirror}
		Assume $\Phi$ is 1-strongly convex w.r.t a certain norm $\| \cdot  \|_\omega$, i.e., 
		\begin{equation} \label{Phi-convex}
		D_\Phi(x,y)  \geq \frac{1}{2} \|x-y\|_\omega^2\,.
		\end{equation}
		Let $\{ u^k\}$ be the solution to \eqref{MD1} with the initial $u^0=u_0$.  Then we have
		\begin{equation} \label{ineq2}
		f\left( \frac{1}{K} \sum_{k=0}^{K-1} u^k \right) - f(u^*) \leq \frac{1}{\etaM K } D_\Phi(u^*, u_0)  + \sumK \frac{\etaM}{2} \|\nabla f(u^k)+A^\top c(u^k)\|_{\omega,*}^2\,,
		\end{equation}
		where $\|\cdot\|_{\omega,*}$ is the dual norm of $\| \cdot  \|_\omega$. 
	\end{theorem}
	
	\begin{proof}
		We mimic the proof of Theorem~\ref{thm:global-cont-mirror}. Consider 
		\begin{align*}
		D_\Phi(u^*, u^\kp) - D_\Phi(u^*, u^k) = \Phi(u^k ) - \Phi(u^\kp) + \nabla \Phi(u^k)^\top (u^*- u^k) - \nabla \Phi(u^\kp)^\top(u^* - u^\kp)\,.
		\end{align*}
		Plugging in the relation \eqref{MD1} and using $A u^*=A u^k=Au^{k+1}=b$, we have 
		\begin{align}
		&\quad D_\Phi(u^*, u^\kp) - D_\Phi(u^*, u^k) \nonumber\\
		&= \Phi(u^k ) - \Phi(u^\kp) + \nabla \Phi(u^\kp)^\top(u^\kp- u^k) + \etaM \nabla f(u^k)^\top(u^*- u^k)\nonumber\\
		&\leq - D_\Phi(u^\kp, u^k) + \etaM (\nabla f(u^k)+A^\top c(u^k))^\top(u^k - u^\kp) + \etaM[f(u^*) - f(u^k)]. \label{MD_con_1}
		\end{align}
		Using the fact that $\Phi$ is 1-strongly convex w.r.t norm $\| \cdot  \|_\omega$, we have 
		\begin{align*}
		- D_\Phi(u^\kp, u^k) + \etaM (\nabla f(u^k)+A^\top c(u^k))^\top(u^k - u^\kp)  \leq  \frac{\etaM^2}{2}  \| \nabla f(u^k)+A^\top c(u^k)\|_{\omega,*}^2,
		\end{align*}
		where we have used the Young's inequality for the term $\etaM (\nabla f(u^k)+A^\top c(u^k))^\top(u^k - u^\kp)$.
		Therefore, we have
		\[
		f(u^k) - f(u^*) \leq \frac{1}{\etaM} \left[ D_\Phi(u^*, u^k) - D_\Phi(u^*, u^\kp) \right] + \frac{\etaM}{2} \| \nabla f(u^k)+A^\top c(u^k)\|_{\omega,*}^2\,.
		\]
		Summing from $k=0$ to $K-1$ and dividing by $K$ give rise to \eqref{ineq2}. 
	\end{proof}
	The inequality~\eqref{ineq2} is still valid if $A^\top c(u^k)$ is removed, and it reduces to the standard convergence result with bounded gradient~\cite{beck2003mirror}. However, we add $A^\top c(u^k)$ here because $\nabla f(u^k)+A^\top c(u^k)\rightarrow 0$, while $\nabla f(u^k)$ may not.
	
	
	\begin{theorem}[Global convergence for variable-metric \eqref{qNewton1}] \label{thm-qNewton1}
		Assume $\Phi$ is 1-strongly convex w.r.t a certain norm $\| \cdot  \|_\omega$ and $\nabla^2\Phi$ is $L$-Lipschitz, i.e., 
		\begin{equation*}
		\|\nabla^2\Phi(x) - \nabla^2 \Phi(y)\|_2 \leq L \|x-y\|_2\,.
		\end{equation*}
		Then the solution $\{ u^k\}$ to \eqref{qNewton1} with initial $u^0=u_0$ satisfies 
		\begin{align}
		f\left( \sumK u^k \right) - f(u^*)  & \leq \frac{1}{\etaN K} D_\Phi(u^*, u_0)  +   \sumK \frac{\etaN}{2} \|\nabla f(u^k)+A^\top c(u^k)\|_{\omega,*}^2  \nonumber
		\\   
		&+{L\etaN\over 2} \sumK  \left[   \|\nabla^2 \Phi(u^k)^{-1} (\nabla f(u^k)+A^\top c(u^k))\|_2^2 \|u^* - u^{k+1}\|_2  \right]
		\,.  \label{ineq3} 
		\end{align}
	\end{theorem}
	\begin{proof}
		Here we follow the approach in the proof of Theorem~\ref{thm-dis-mirror} but tailor the details according to the update rule \eqref{qNewton1}. First we write
		\begin{align}
		&\quad D_\Phi(u^*, u^\kp) - D_\Phi(u^*, u^k) \nonumber\\
		&= \Phi(u^k ) - \Phi(u^\kp) + \nabla \Phi(u^k)^\top (u^*- u^k) - \nabla \Phi(u^\kp)^\top(u^* - u^\kp)
		\nonumber
		\\ & = -D_\Phi(u^{k+1},u^k)+[\nabla \Phi(u^k) - \nabla \Phi(u^\kp)]^\top (u^* - u^\kp) \nonumber
		\\ & \leq [\nabla \Phi(u^k) - \nabla \Phi(u^\kp)]^\top (u^* - u^\kp) \,, \label{0412}
		\end{align}
		where the inequality is due to  convexity of $\Phi$. Next, compute the difference
		\begin{align}
		&\quad\nabla \Phi(u^k) - \nabla \Phi(u^\kp)\nonumber\\
		 &=  -\int_0^1 \nabla^2 \Phi(u^k + t(u^\kp - u^k)) (u^\kp - u^k) \rd t \nonumber
		\\ & =  \etaN \int_0^1 \nabla^2 \Phi(u^k + t(u^\kp - u^k)) \nabla^2\Phi(u^k)^{-1} (\nabla f(u^k)+A^\top c(u^k)) \rd t  \label{0413}
		\\ & = \etaN (\nabla f(u^k)+A^\top c(u^k)) - \etaN \underbrace{\int_0^1 \left[  I - \nabla^2 \Phi(u^k + t(u^\kp - u^k)) \nabla^2\Phi(u^k)^{-1}    \right] (\nabla f(u^k)+A^\top c(u^k)) \rd t}_{\mathcal{A}}. \nonumber
		\end{align}
		Plugging \eqref{0413} into \eqref{0412} gives
		\begin{align} \label{eqn0608}
		\quad D_\Phi(u^*, u^\kp) - D_\Phi(u^*, u^k) & \leq \etaN (\nabla f(u^k)+A^\top c(u^k))^\top(u^k - u^\kp + u^* - u^k) - \etaN\mathcal{A}^\top(u^*-u^\kp) \nonumber
		\\ & \leq \etaN (\nabla f(u^k)+A^\top c(u^k))^\top(u^* - u^\kp) + \etaN[f(u^*) - f(u^k)] \nonumber
		\\ & \hspace{4cm} - \eta \mathcal A^T(u^*-u^\kp)\,.
		\end{align}
        Because $\nabla^2 \Phi$ is $L$-Lipschitz, we have
		\begin{align} 
		\|\mathcal{A}\|_2 &\leq \frac{L}{2} \|u^\kp - u^k\|_2 \|\nabla^2 \Phi(u^k)^{-1} (\nabla f(u^k)+A^\top c(u^k))\|_2 \nonumber\\
		&= \frac{L\etaN}{2} \|\nabla^2 \Phi(u^k)^{-1} (\nabla f(u^k)+A^\top c(u^k))\|_2^2.\label{Aineq}
		\end{align}
		Therefore, similarly to the previous theorem, we have
		\begin{align*}
		&\quad f(u^k) - f(u^*)  \\
		& \leq \frac{1}{\etaN} \left[D_\Phi(u^*, u^k) -  D_\Phi(u^*, u^\kp)  \right] + (\nabla f(u^k)+A^\top c(u^k))^\top (u^k-u^\kp)- \mathcal{A}^\top(u^*-u^\kp)	\\ 
		& = \frac{1}{\etaN} \left[D_\Phi(u^*, u^k) -  D_\Phi(u^*, u^\kp)  \right] +\frac{\etaN}{2} \| \nabla f(u^k)+A^\top c(u^k)\|_{\omega,*}^2
- \mathcal{A}^\top(u^*-u^\kp)\\ 
& \leq  \frac{1}{\etaN} \left[D_\Phi(u^*, u^k) -  D_\Phi(u^*, u^\kp)  \right]  +\frac{\etaN}{2} \| \nabla f(u^k)+A^\top c(u^k)\|_{\omega,*}^2
		\\  &\hspace{6cm}+ 
		\frac{L\etaN}{2} \|\nabla^2 \Phi(u^k)^{-1} (\nabla f(u^k)+A^\top c(u^k))\|_2^2 \|u^* - u^{k+1}\|_2\,.
		\end{align*}
		Summing from $k=0$ to $K-1$ and dividing it by $K$, we arrive at \eqref{ineq3}. 
	\end{proof}
	
	Comparing \eqref{eqn0608} to \eqref{MD_con_1}, one sees that the difference lies in the additional term $\mathcal A$, which leads to the different results in \eqref{ineq2} and~\eqref{ineq3}. It is not obvious to say which one converges faster just based on this comparison. However, in practice, \eqref{MD1} is superior to \eqref{qNewton1} mainly due to its flexibility in treating constraints. In particular, for cases when $u$ has some bound constraints, such as non-negativity, one can directly build such constraint in $\Phi$ for \eqref{MD1} and the resulting solution is automatically bound preserving. Whereas in \eqref{qNewton1}, there is no such a guarantee. See numerical examples in Figs.~\ref{fig:agg3} and \ref{fig:CH} for an evidence.

	As a side note, we can extend the result in Theorem~\ref{thm-qNewton1} to general quasi-Newton methods:
	\begin{equation} \label{qNewton2}
	u^{k+1} = u^k - \eta B_k^{-1} (\nabla f(u^k) + A^\top c(u^k))\,,
	\end{equation}
	where $B_k$ is an approximated Hessian that satisfies 
	\begin{equation} \label{HB}
	B_{k+1} (u^{k+1} - u^k) = \nabla f(u^{k+1}) - \nabla f(u^k)\,,
	\end{equation}
	and $c(u^k)$ is again the to-be-determined vector that warrants $A u^{k+1} = b$.
	
	
	\begin{theorem}[Global convergence for quasi-Newton \eqref{qNewton2}] \label{thm-qNewton2}
		Let $\{ u^k\}$ be the solution to \eqref{qNewton2}--\eqref{HB} with initial guess $u_0$. If ${B_k}$ satisfies
		\begin{equation} \label{Bineq}
		\|B_{k+1} B_k^{-1} - I \|_2 \leq \eta L 
		\end{equation} 
		for some constant $L$, which is independent of $\eta$ and $k$,  then we have 
		\begin{align} \label{ineq4}
		f\left( \sumK u^k \right) - f(u^*) & \leq \frac{1}{\eta K} D_f(u^*, u_0)  + 
		\etasumK \left[   \|\nabla f(u^k) + A^\top c(u^k)\|^2_{B_k^{-1}} \right. \nonumber 
		\\ &\qquad \left. + L \| \nabla f(u^k) + A^\top c(u^k)\|_2 \| u^*-u^\kp\|_2    \right]    \,.
		\end{align}
	\end{theorem}
	
\begin{proof}
Because there is no function $\Phi$, we use the objective function $f$ to define $D_f$ and control the distance between current iteration and the optimal solution. More precisely, we consider
	\begin{align}
		D_f(u^*, u^\kp) - D_f(u^*, u^k) & = [\nabla f(u^k) - \nabla f(u^{k+1})]^\top (u^* - u^\kp)-D_f(u^{k+1},u^{k}) \nonumber
		\\ & \leq [B_\kp (u^k - u^\kp)]^\top (u^* - u^\kp)  \nonumber
		\\ & = [ B_k(u^k - u^\kp) + (B_\kp - B_k)(u^k - u^\kp)]^\top(u^* - u^\kp)   \label{0414}
		\end{align}
		where the first inequality come from \eqref{HB}. Using \eqref{qNewton2}, we see that $B_k(u^k - u^\kp) =  \eta ( \nabla f(u^k) + A^\top c(u^k))$. Thus, \eqref{0414} becomes
		\begin{align*}
		& D_f(u^*, u^\kp) - D_f(u^*, u^k) 
		\\ & \leq  \eta (\nabla f(u^k) + A^\top c(u^k))^\top(u^*-u^k) + \eta (\nabla f(u^k) + A^\top c(u^k))^\top(u^k - u^\kp)
		\\ & \hspace{4cm} + [(B_\kp - B_k) (u^k - u^\kp)]^\top (u^*-u^\kp)
		\\ & \leq \eta [f(u^*) - f(u^k)] + \eta^2 (\nabla f(u^k) + A^\top c(u^k))^\top B_k^{-1} (\nabla f(u^k) + A^\top c(u^k)) 
		\\ & \hspace{4cm} + \eta^2 L \| \nabla f(u^k) + A^\top c(u^k)\|_2 \| u^* - u^\kp \|_2\,,
		\end{align*}
		where we have used the convexity of $f$ and continuity of $B_k$ in \eqref{Bineq}. Then \eqref{ineq4} follows from summing the following inequality over $k$ and dividing by $K$.
	\end{proof}

	\subsection{Local convergence}
	In this section, we show the local convergence of \eqref{MD1}. For notation brevity, we omit the subscript in $\eta$ and simply write \eqref{MD1} as
	\begin{equation} \label{MD3}
	\nabla \Phi (u^\kp) - \nabla \Phi(u^k) = - \eta (\nabla f (u^k)+A^\top c(u^k))\,.
	\end{equation}
	First we have the following proposition showing that if the iteration step $\eta$ is properly chosen, the objective function sufficiently decreases along the flow. This mimics the first stage of Newton's method. 
	\begin{proposition}
		Assume $\Phi$ is 1-strongly convex with respect to the standard norm, i.e., $\nabla^2 \Phi \succeq I$, and $\nabla f$ is $L-$Lipschitz. If $\eta $ is chosen by $\eta \leq \min \{1,  \frac{2}{L} (1-\alpha)\}$ for $\alpha \in (0, 0.5)$, then we have the following sufficient descent condition 
		\begin{equation} \label{decrease}
		f(u^\kp) \leq f(u^k) + \alpha \nabla f(u^k)^\top (u^\kp - u^k)\,.
		\end{equation}
		
	\end{proposition}
	\begin{proof}
Recall the definition of $\Delta u$ in \eqref{diffu}, then we have 
		\begin{align}
		f(u^\kp) - f(u^k) &= \int_0^1 \nabla f(u^k + t \Delta u)^\top \Du \rd t  \nonumber
		\\ & =  \nabla f(u^k)^\top \Du + \int_0^1 [\nabla f(u^k + t \Delta u)  - \nabla f (u^k)]^\top \Du \rd t  \nonumber
		\\ & \leq \alpha \nabla f(u^k)^\top \Du  + (1-\alpha) \nabla f(u^k)^\top \Du  + \frac{L}{2} \|\Du\|^2\,.  \label{0420}
		\end{align}
		From \eqref{MD3}, one sees that 
		\begin{align*}
		\nabla f(u^k)^\top \Du=(\nabla f(u^k)+A^\top c(u^k))^\top \Du = -\frac{1}{\eta} ( \nabla \Phi(u^\kp) - \nabla \Phi(u^k))^\top \Du \leq - \frac{1}{\eta } \|\Du\|^2 \,.
		\end{align*}
		Plugging it into \eqref{0420}, we have
		\begin{align*}
		f(u^\kp) - f(u^k)  \leq \alpha \nabla f(u^k)^\top \Du  + \left(  \frac{L}{2} - \frac{1-\alpha}{\eta }  \right) \|\Du\|^2\,.
		\end{align*}
		Then choosing $\eta \leq \frac{2}{L} (1-\alpha)$ makes $\frac{L}{2} - \frac{1-\alpha}{\eta } \leq 0$ and therefore \eqref{decrease} holds. 
	\end{proof}

	Next, we intend to show that after sufficiently large number of iterations, one may reach the second stage of Newton type methods and result in superlinear convergence. 
	\begin{lemma} \label{lemma-eta}
		If
		\begin{equation} \label{Ak}
		G_k := \int_0^1 \nabla^2 \Phi(u^k + s \Du) \rd s 
		\end{equation}
		satisfies the Dennis-Mor\'e condition \cite{DM74}:
		\begin{equation} \label{DMc}
		\frac{\| (G_k - \nabla^2 f(u^*)) (u^\kp - u^k)\|}{ \|  u^\kp - u^k\|} \rightarrow 0, \qquad \text{as} ~ k \rightarrow \infty\,,
		\end{equation}
		and $\nabla^2 f$ is $L_2-$Lipschitz,   then $\eta =1$ in \eqref{MD3} satisfies the sufficient descent condition \eqref{decrease} for  sufficiently large $k$.
	\end{lemma}
	\begin{proof} We have 
		\begin{align*}
		&\quad~f(u^\kp) - f(u^k) 
		\\ & = \nabla f(u^k)^\top \Du + \Du^\top  \int_0^1 (1-t) \nabla^2 f(u^k + t \Delta u)  \rd t \Du
		\\ & = \half  \nabla f(u^k)^\top \Du + \half \Du^\top [\nabla \Phi(u^k) - \nabla \Phi (u^k + \Du)] + \Du^\top  \int_0^1 (1-t) \nabla^2 f(u^k + t \Delta u)  \rd t \Du
		\\ & = \half  \nabla f(u^k)^\top \Du - \!\half \Du^\top  \left[  G_k -\! \nabla^2 f(u^*)\right] \Du + \Du^\top  \int_0^1 (1\!-\! t) \left[  \nabla^2 f(u^k + t \Delta u) \!-\! \nabla^2 f(u^*) \right]  \rd t \Du\,.
		\end{align*}
		From \eqref{DMc}, one has $\|\left[  G_k -\! \nabla^2 f(u^*)\right] \Du\| \leq o(\| \Du\|)$. Therefore,
		\begin{align*}
		f(u^\kp) - f(u^k)  \leq  \half  \nabla f(u^k)^\top \Du + o(\|\Du\|^2) + \frac{L_2}{2} \|\Du\|^2 \|u^\kp - u^*\|\,.
		\end{align*}
		For $k$ sufficiently large, $\Du$ is sufficiently small, so the descent condition \eqref{decrease} will be satisfied. 
	\end{proof}
	
	Once $\eta$ is chosen to be 1, the superlinear convergence of \eqref{MD3} can be obtained.
	\begin{theorem}[Local superlinear convergence of \eqref{MD3}]
		If $G_k$ defined  in \eqref{Ak} satisfies the Dennis-Mor\'e condition \eqref{DMc}, and assume that around $x^*$, $f$ is $\beta$-strongly convex and $\nabla^2 f$ is $L$-Lipschitz. Then the mirror descent \eqref{MD3} converges superlinearly, i.e., $\| u^\kp - u^*\| \leq o(\| u^k - u^*\|)$.
	\end{theorem}
	\begin{proof}
		From Lemma~\ref{lemma-eta}, the unit step length is allowed after sufficiently many iterations, and therefore we have 
		\[
		\nabla \Phi(u^\kp) - \nabla \Phi(u^k) = - (\nabla f(u^k)+A^\top c(u^k))\,,
		\]
		which can be rewritten as 
		\[
		G_k (u^\kp - u^k) = -  (\nabla f(u^k)+A^\top c(u^k))\,,
		\]
		where $G_k$ is defined in \eqref{Ak}. Then we have 
		\begin{align}
		    & \quad~(G_k - \nabla^2 f(u^*)) (u^\kp - u^k) \nonumber
		    \\ & = -(\nabla f(u^k) + A^\top c(u^k)) - \nabla^2f(u^*)(u^\kp-u^k) \nonumber
		   \\ &= \nabla f(u^\kp)- \nabla f(u^k)  - \nabla^2f(u^*)(u^\kp-u^k) - (\nabla f(u^\kp) + A^\top c(u^k)). \label{eqn:0609}
		\end{align}
		From the Lipschitz continuity of $\nabla^2 f$ around $u^*$, we have 
 		$\| \nabla f(u^\kp)-\nabla f(u^k) - \nabla^2 f(u^*)(u^\kp-u^k)\|/\|u^\kp - u^k\| \rightarrow 0$ as $k \rightarrow \infty$. Then using the Dennis-Mor\'e condition \eqref{DMc}, \eqref{eqn:0609} implies 
 		\begin{equation*}
 		    \lim_{k \rightarrow \infty} \frac{\|  \nabla f(u^\kp) + A^\top c(u^k) \|}{\| u^\kp - u^k\|} = 0\,,
 		\end{equation*}
 		which readily leads to 
        \begin{equation*}
 		    \lim_{k \rightarrow \infty} \frac{\| \nabla f(u^\kp) + A^\top c(u^k) - \nabla f(u^*) - A^\top c(u^*)\|}{\| u^\kp - u^k\|} = 0\,.
 		\end{equation*}
		The above equation also implies 
		   \begin{equation*}
 		    \lim_{k \rightarrow \infty} \frac{ \langle \nabla f(u^\kp) + A^\top c(u^k) - \nabla f(u^*) - A^\top c(u^*), (u^\kp - u^*)/\|u^\kp-u^*\|\rangle }{\| u^\kp - u^k\|} = 0\,.
 		\end{equation*}
 		Then using the fact that $Au^\kp = A u^*$, it reduces to 
	   \begin{equation*}
 		    \lim_{k \rightarrow \infty} \frac{ \langle \nabla f(u^\kp) - \nabla f(u^*) , u^\kp - u^*\rangle }{\| u^\kp - u^k\| \|u^\kp - u^*\|} = 0\,.
 		\end{equation*}
 		Since $\langle \nabla f(u^\kp) - \nabla f(u^*) , u^\kp - u^*\rangle \geq \frac{\beta}{2} \|u^\kp - u^*\|^2$ for sufficiently large $k$ and $\|u^\kp-u^k\| \leq \|u^\kp - u^*\|+\|u^k-u^*\|$, we have 
 		\[
 		\lim_{k \rightarrow \infty} \frac{\beta}{2} \frac{\|u^\kp-u^*\|}{\|u^\kp-u^*\|+\|u^k-u^*\|}=0\,,
 		\]
	    which implies $\lim_{k\rightarrow \infty} \frac{\|u^\kp-u^*\|}{\|u^k-u^*\|}=0$. 
	\end{proof}
	
	We also mention that in general \eqref{DMc} is not satisfied, so instead of having the superlinear convergence, we will have a linear convergence but with an increased rate as compared to the standard gradient descent. More specifically, we have the following theorem.

    \begin{theorem} {Let $\Phi$ be 1-strongly convex with respect to $\|\cdot\|_w$, and the sequence $\{u^k\}$ is  obtained from \eqref{MD3}. Also, Assume that $D_f(u^*,u^k)\geq \mu D_\Phi(u^*,u^k)$. Then we have 
    $$D_\Phi(u^*,u^{k+1})\leq (1-\eta \mu)D_{\Phi}(u^*,u^k),$$ if $\eta\leq 2(f(u^k)-f(u^*))/\|\nabla f(u^k)\|_{w,*}^2$. }
    \end{theorem}
    \begin{proof}
    From the definition of Bregman divergence, we have \begin{align*}
      & \quad~ D_\Phi(u^*,u^{k+1})-D_\Phi(u^*,u^k)\\
      &=[\Phi(u^*)-\Phi(u^{k+1})-\langle\nabla \Phi(u^{k+1}),u^*-u^{k+1}\rangle] - [\Phi(u^*)-\Phi(u^{k})-\langle\nabla \Phi(u^{k}),u^*-u^{k}\rangle] \\
    &= -[\Phi(u^\kp)-\Phi(u^{k})-\langle\nabla \Phi(u^{k}),u^\kp-u^{k}\rangle]+\eta\langle\nabla f(u^k),u^*-u^\kp\rangle\\
     &\leq -\|u^k-u^\kp\|_w^2/2 +\eta\langle \nabla f(u^k),u^k-u^\kp\rangle-\eta(f(u^k)-f(u^*)) -\eta D_f(u^*,u^k)\\
     &\leq \eta^2\|\nabla f(u^k)\|_{w,*}^2/2-\eta(f(u^k)-f(u^*))-\eta \mu D_\Phi(u^*,u^k)
    \end{align*}
    Therefore, if we choose $\eta\leq 2(f(u^k)-f(u^*))/\|\nabla f(u^k)\|_{w,*}^2$, then we have 
    \begin{align*}
    D_\Phi(u^*,u^{k+1})\leq (1-\eta \mu)D_{\Phi}(u^*,u^k).
    \end{align*}
    The theorem is proved.
    \end{proof}
  This theorem is consistent with Theorem~\ref{thm2}. When $\Phi$ is properly chosen, it will mitigate the ill-conditioning inherited from $f$ in the sense that $\mu$ is increased, and therefore leads to a much improved rate of convergence.   
	
\section{Applications and numerical experiments}
Apart from the examples mentioned in \cite{Ying20}, we consider two additional applications of the mirror descent \eqref{MD1} in evolutionary PDEs: the Wasserstein gradient flow and Cahn-Hillard equation with degenerate mobility. In particular, viewing the Wasserstein gradient flow as a weighted $H^{-1}$ gradient flow and using a minimizing movement scheme, we obtain an ill-conditioned  optimization problem. The same problem is encountered in the Cahn-Hillard equation when the mobility is degenerate. In both cases, our mirror descent can provide  preconditioning mechanisms while preserving the bounds of the solution (e.g., positivity) and mass conservation.  
	
\subsection{Wasserstein gradient flow}
Let's consider the following Wasserstein gradient flow
\begin{equation} \label{gd}
\partial_t \rho(t,x) = - \grad_{\dW} \energy(\rho(t,x)) :=  \grad \cdot \left( \rho(t,x) \grad  \variErho (\rho(t,x)) \right),
\end{equation}
where $\Wass_2$ is the quadratic Wasserstein metric and $\delta$ denotes the first variation. Here $\rho(t,x)$ with $x\in\Omega\subset \mathbb{R}^n$ is the particle density function, and energy $\energy(\rho(t,x))$ takes the form
	\begin{equation*} 
\energy (\rho(t,x)) = \int_{\Omega} \left[U(\rho(t,x)) + V(x) \rho(t,x) \right]\rd x + \half \int_{\Omega \times \Omega} W(x-y) \rho(t,x) \rho(t,y) \rd x \rd y\,.
	\end{equation*}
The no-flux boundary condition $\rho(t,x) \nabla \delta \energy \cdot \hat n \big|_{\partial \Omega} = 0$ is imposed to ensure the mass conservation. This equation has diverse applications in physics and biology, such as granular materials \cite{CMV03}, chemotaxis \cite{KellerSegel1971}, animal swarming \cite{CFTV10, BCCD16}, and many others.

	
Numerically solving \eqref{gd} has been quite challenging to satisfy three desired properties: non-negativity, mass conservation, and energy dissipation. Besides the Eulerian and Lagrangian methods that have been developed in the literature, we particularly mention the variational approach following the seminal JKO scheme by Jordan, Kinderlehrer, and Otto \cite{JKO98}. Given a time step $\tau>0$, the JKO scheme recursively defines a sequence $\rho_n(x)$
via a minimizing movement approach. This approach has revolutionized PDE analysis, whereas its impact in numerics has only be revealed recently with the aid of modern optimization algorithms \cite{peyre2015entropic, BCL16, CPSV18, CCWW18, LLW20, JLL20}.

	

	In this paper, we consider a similar but slightly different approach. In particular, we obtain the solution sequence $\{\rho_n\}$, an approximation to the exact solution $\rho_n(x) \approx \rho(n\tau, x)$ as follows: 
	\begin{equation} \label{JKO-H1}
	\rho_0 = \rho_{\textrm{in}}, \quad \rho_{n+1} = \arg\min_{\rho\in \KK} \left\{  \frac{1}{2\tau }\|\rho - \rho_n\|_{\Hnrhoi}^2 +  \energy(\rho) \right\} : = \arg\min_{\rho\in \KK} f(\rho) \,,
	\end{equation}
	where $\|u \|_{\Hnrhoi}^2  = \int_{\RR^d} u(x) \Hnrhoi u(x) \rd x$, and $\KK = \left\{ \rho: \rho \in \mathcal{P}(\Omega), ~
	\int_{\Omega}|x|^2 \rho \,\rd x < + \infty \right\}$. Here $\Hnrho$ is the negative weighted Laplacian $\Hnrho = -\nabla \cdot (\rho_n \nabla )$, and $\Hnrhoi$ is its pseudo-inverse. It has been shown that the weighted $H^{-1}$ norm is a first order approximation to the Wasserstein distance \cite{villani2003topics}, and therefore will not violate the first order accuracy of the JKO scheme \cite{AGSbook}. In view of \eqref{JKO-H1}, one sees that the three desired properties mentioned above are all satisfied. Indeed, the positivity and mass conservation are obtained by requiring the minimizer in $\KK$, the energy dissipation $\energy(\rho_{n+1}) \leq \energy(\rho_n)$ is also immediate since $\rho_{n+1}$ is the minimizer.

	\subsubsection{Mirror descent algorithm}\label{sec:W-MD}
	To solve the optimization problem \eqref{JKO-H1}, a direct projected gradient descent takes the following form
	\begin{equation} \label{GD}
	\rho^\kp = \textrm{proj}_\KK ~\left\{  \rho^k - \eta \left[ \frac{1}{\tau} \Hnrhoi(\rho^k - \rho_n) + \variErho (\rho^k) \right]   \right\}\,.
	\end{equation}
	where $\eta$ is the iteration stepsize and the superscript $k$, which shall not be confused with the subscript $n$, denotes the iteration index. Since the value of $\rho_n$ can be arbitrarily close to zero, $\Hnrho$ is very stiff, and therefore the gradient descent \eqref{GD} will take extremely long time to converge. To this end, we propose the following mirror descent algorithm.
	
	Choosing $\Phi$ in \eqref{MD1} to be
	\begin{equation} \label{W-phi}
	\Phi(\rho) = \frac{1}{2\tau }\|\rho - \rho_n\|_{\Hnrhoi}^2 + \eps \int \rho \log \rho \rd x\,,
	\end{equation}
	then the mirror descent reads
	\[
	\frac{\delta \Phi}{\delta \rho}(\rho^\kp)- \frac{\delta \Phi}{\delta \rho}(\rho^k) = -\eta \left[ \frac{1}{\tau} \Hnrhoi(\rho^k - \rho_n) + \variErho (\rho^k) \right] \,,
	\]
	which simplifies to
	\begin{equation} \label{W-MD1}
	\rho^\kp + \eps \tau \Hnrho \log \rho^\kp = \rho^k + \eps \tau \Hnrho \log \rho^k - \eta \left[ \rho^k - \rho_n + \tau \Hnrho \variErho (\rho^k) \right]\,.
	\end{equation}
	It is important to point out that, thanks to the additional entropy term in \eqref{W-phi}, the positivity of $\rho$ is preserved in \eqref{W-MD1}. Moreover, since $\Hnrho$ preserves mass, i.e., $\int \Hnrho u(x) \rd x = 0$, mass conservation is also guaranteed in \eqref{W-MD1}, that is, $\int \rho^\kp(x) \rd x = \int \rho^k (x) \rd x= \int \rho_n(x) \rd x$. 
	
	In practice, we will further discretize \eqref{W-MD1} in space. Let us consider one dimension for instance. Denote $[x_L,x_R]$ as the computational domain and $\Delta x$ the spatial grid. Choose $x_j=x_L + (j-\half)\Delta x$, and denote
	\begin{align*} 
	\rho_{j} \approx \rho(x_j), \quad \rho_{n,j} \approx \rho( t_n, x_j),  \quad 1 \leq j \leq N_x, ~ n \in \mathbb{N}_+\,,
	\end{align*}
	where $t_n = n \tau$, and $N_x \Delta x = x_R - x_L$. First we discretize $\Hnrho$, and denote its  discrete counterpart as $\Drhon$. Then we propose 
	\begin{align} \label{disD}
	(\Drhon u)_j = \left\{ \begin{array}{ll} 
	-\frac{1}{\Delta x^2} \left[ \frac{\rho_{n,j} + \rho_{n,\jp} }{2}  u_\jp - \frac{\rho_{n,\jp} + \rho_{n,j} }{2} u_j   \right]\,, & j = 1
	\\ -\frac{1}{\Delta x^2} \left[ \frac{\rho_{n,j} + \rho_{n,\jp} }{2}  u_\jp - \frac{\rho_{n,\jp} + 2\rho_{n,j} + \rho_{n,\jm}}{2} u_j + \frac{\rho_{n,j}  + \rho_{n,\jm}}{2} u_\jm   \right]\,, & 2\leq j \leq N_x-1
	\\  -\frac{1}{\Delta x^2} \left[ - \frac{\rho_{n,j} + \rho_{n,\jm}}{2} u_j + \frac{\rho_{n,j}  + \rho_{n,\jm}}{2} u_\jm   \right] \,, & j = N_x
	\end{array} \right.\,.
	\end{align}
	Note specifically that for $j = 1$ and $j=N_x$, our discretization takes into account the boundary condition in \eqref{gd}. Indeed, if either $\rho\big|_{\partial \Omega} = 0$ or $u\cdot \hat n \big|_{\partial \Omega} = 0$, then at discrete level on the left boundary, we have $\rho_{n,0}+\rho_{n,1} = 0$ or $u_1-u_0 = 0$ correspondingly, and in either the second line of \eqref{disD} reduces to the first line. Same arguments applies to the right boundary.  
	As a result, $\Drhon$ preserves the mass, i.e., ${\mathsf 1}^\top \Drhon = {\mathsf 0}$. 
	
	Denote 
	\[
	{\bfrho} = (\rho_1, \rho_2, \cdots, \rho_{N_x})^T, \quad 
	{\bfrho}_n = (\rho_{n,1}, \rho_{n,2}, \cdots, \rho_{n, N_x})^T\,,
	\]
	we can rewrite \eqref{W-MD1} in the discrete form 
	\begin{equation} \label{W-MD2}
	\bfrho^\kp + \eps \tau \Drhon \log \bfrho^\kp = \bfrho^k + \eps \tau \Drhon \log \bfrho^k - \eta \left[\bfrho^k - \bfrho_n + \tau \Drhon \variErho (\bfrho^k) \right]\,.
	\end{equation}
	Then the remaining task is to solve the nonlinear equations for $\bfrho^\kp$, for which we use the Newton's method. Let $\bfy = \log \bfrho$ and define $\bfh$ to be 
    $
	\bfh (\bfrho) = e^{\bfy} + \eps \tau \Drhon \bfy - \bfb \,,
	$
	where $\bfb =\bfrho^k + \eps \tau \Drhon \log \bfrho^k - \eta [\bfrho^k - \bfrho_n + \tau \Drhon \delta \energy(\bfrho^k)] $. Then the Newton's method takes the form
	\begin{equation} \label{eqn423}
	\bfy^\llp = \bfy^\ls - \matA_l^{-1} \bfh(\bfy^\ls)\,,
	\end{equation}
	where 
    $
	\matA_l =  \textrm{diag} [e^{\bfy^{(l)}} ] + \eps \tau \Drhon\,.
	$
	Note that since the components of $\bfy$ can vary drastically, $\matA$ will be ill-conditioned, and therefore the computation $\matA_l^{-1} \bfh(\bfy^\ls)$ in \eqref{eqn423} may be susceptible to errors. To fix this issue, we propose the following preconditioner $\matP_l = \textrm{diag}[e^{-\bfy^\ls}]$, and rewrite \eqref{eqn423} into
	\begin{equation} \label{eqn0523}
	\bfy^\llp = \bfy^\ls - (\matP_l \matA_l) ^{-1}[\matP_l \bfh(\bfy^\ls)] \,.
	\end{equation}
	Note that since $\matP_l$ is a diagonal matrix, the preconditioner is cheap to apply. 
	
	In summary, we have the following algorithms.
	
	\begin{algorithm}[H]
		\SetAlgoLined
		{\bf Input} $\bfrho_n$,$\Drhon$, $\text{Iter}_{\max}$, $\Delta x$, $\tau$ \\ 
		{\bf Output} $\bfrho_{n+1}$\\
		$\bfrho^0 = \bfrho_n$, $\bfy^0 = \log \bfrho^0$\\
		$k=0$\\
		\While{$k \leq \text{Iter}_{\max}$ and stopping criteria is not achieved}{
			$\bfy = \bfy^k$,   \\ 
			\While{$\textrm{error}>1e-6$}
			{ $\matA =  \textrm{diag} [e^{\bfy }] + \eps \tau \Drhon$, $\matP = \textrm{diag}[e^{-\bfy}]$ 
				\\$\tilde\bfy = \bfy - (\matP \matA) ^{-1}[\matP \bfh(\bfy)]$
				\\ $\textrm{error} =  \|\tilde\bfy - \bfy\|/\| \bfy\|$
				\\ $\bfy = \tilde \bfy$
			} 
			$k = k+1$, $\bfy^\kp = \tilde \bfy$
		}
		$\bfrho_{n+1} = e^{\bfy^\kp}$
		\caption{Mirror descent algorithm for \eqref{JKO-H1}}
	\end{algorithm}

	\begin{algorithm}[H]
		\SetAlgoLined
		{\bf Input} $\bfrho_0 = \bfrho_\textrm{in}$ \\ 
		{\bf Output} $\bfrho_n$ for $1\leq n \leq N_t$ \\
		\For{$n<=N_t$}{
			Apply Algorithm 1 to $\bfrho_n$ to get $\bfrho_{n+1}$
		}
		\caption{Variational scheme for \eqref{gd}}
	\end{algorithm}

	\subsection{Numerical examples}
We consider two examples of \eqref{gd} and demonstrate the efficiency of mirror descent. In both examples, we stop Algorithm 1 when the relative error is less than a preset tolerance, i.e., 
	\begin{equation} \label{tol}
	   \frac{\|\bfrho^{k+1} - \bfrho^k\|}{\| \bfrho^k \|} \leq {\text{Tol}}, \qquad {~where~} \quad \bfrho^k = \exp^{{\bf y}^k}.
	\end{equation}

	\subsubsection{Porous medium equation}
	We first consider the porous medium equation
	\begin{equation} \label{PMEeqn}
	\partial_t \rho = \Delta \rho^{m}\,, \quad m>1\,,
	\end{equation}
	which can be seen as the Wasserstein gradient flow of $\energy(\rho) = \int \frac{1}{m-1} \rho^m \rd x$. A well-known family of exact solutions is given by the Barenblatt profiles, which are densities of the form
	\begin{equation} \label{eqn:Barenblatt}
	\rho(x,t) = (t+t_0)^{-\frac{1}{m+1}} \left( C - \alpha \frac{m-1}{2m(m+1)}  x^2 (t+t_0)^{-\frac{2}{m+1}} \right)_{+}^{\frac{1}{m-1}} , \qquad \text{ for } C, t_0 > 0 .
	\end{equation}
	In our tests, we choose $m=2$, $t_0 = 10^{-3}$, and $C = 0.8$. The results using our mirror descent algorithm are gathered in Fig.~\ref{fig:porous-MD}. On the left, the numerical solutions are compared to the analytical formulas, and good agreement is demonstrated. In the middle, we have shown the number of iterations needed in Algorithm 1 within each outer time step, and one sees that around the same iterations are needed for a given tolerance Tol= $10^{-8}$. On the right, we plot the relative error \eqref{tol} versus the iteration in the first outer time step. The decay of relative error behaves quite similar at later times. 
	
	It is interesting to mention that, we have also implemented the variable metric method \eqref{qNewton1}, or equivalently \eqref{pNewton}. In particular, we choose $\Phi(\rho) = \frac{1}{2\tau} \|  \rho - \rho_n \|_{\Hnrhoi}$ instead of \eqref{W-phi}, then \eqref{qNewton1} becomes 
	\begin{equation} \label{variM0}
	\rho^{k+1} = \rho^k - \eta \Hnrho (\nabla f(\rho^k) + A^\top c(\rho^k))\,,
	\end{equation}
	where $f(\rho)$ is defined in \eqref{JKO-H1}, $\energy(\rho) = \int \frac{1}{m-1} \rho^m \rd x$, and $A$ is a all one row vector that encodes the mass conservation of $\rho$. As explained before, $\Hnrho$ preserves the mass exactly, therefore $c(\rho^k) \equiv 0$ here and the algorithm reduces to
	\[
	\rho^\kp = \rho^k - \eta (\rho^k - \rho_n) - \eta \Hnrho \variErho (\rho^k)\,.
 	\]
	Compare it to \eqref{W-MD1}, we see that the major difference is that here there is no mechanism to automatically guarantee the positivity. However, with a proper choice of iteration step $\eta$, the positivity may still be preserved. In this specific example, we  apply \eqref{variM0} with the same parameters as in the mirror descent algorithm, and we have obtained exactly the same behavior of the solution as displayed in Fig.~\ref{fig:porous-MD}, so the plots are omitted. 
	
	\begin{figure}[!ht]
		\includegraphics[width=0.32\textwidth]{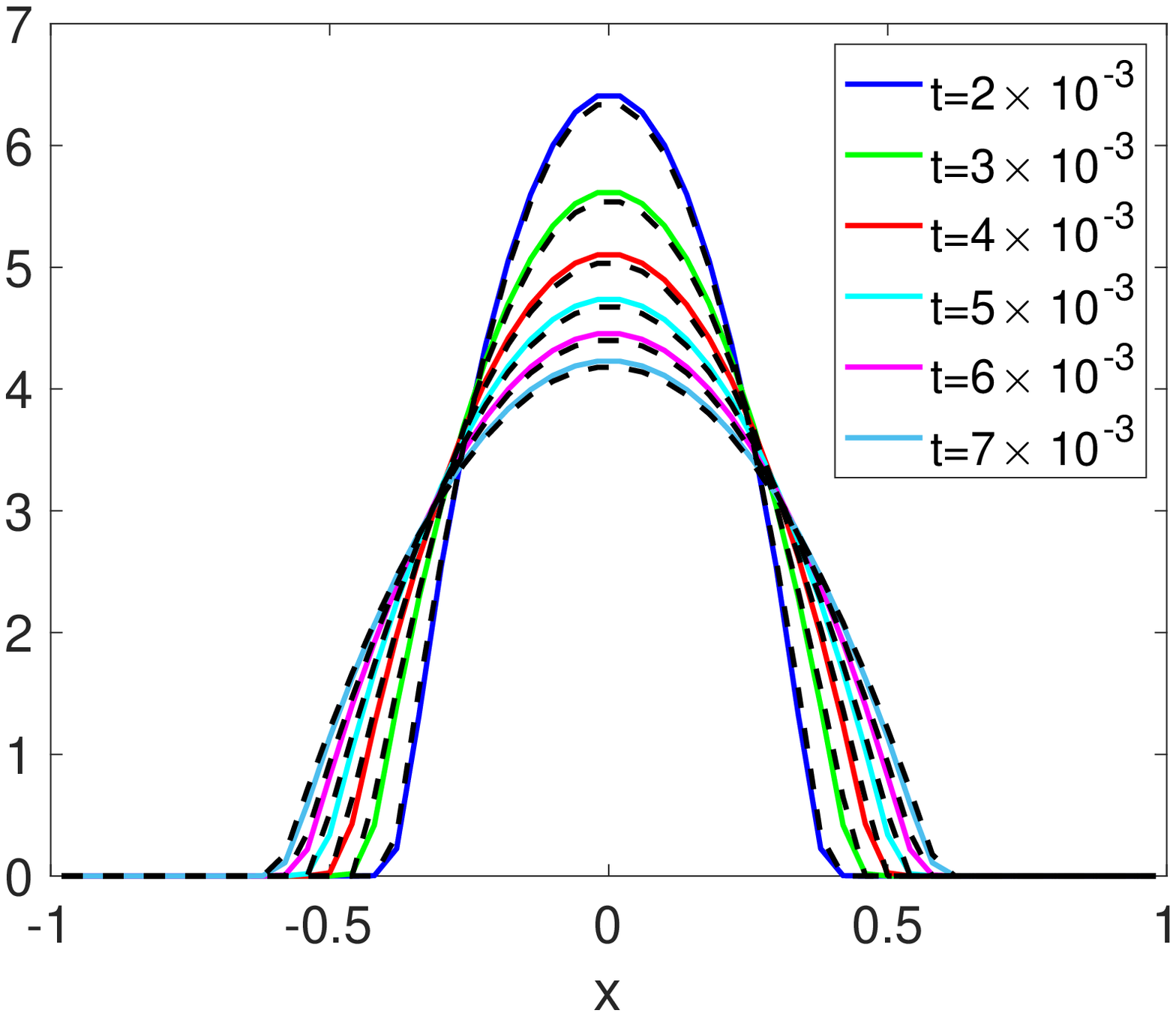}
		\includegraphics[width=0.32\textwidth]{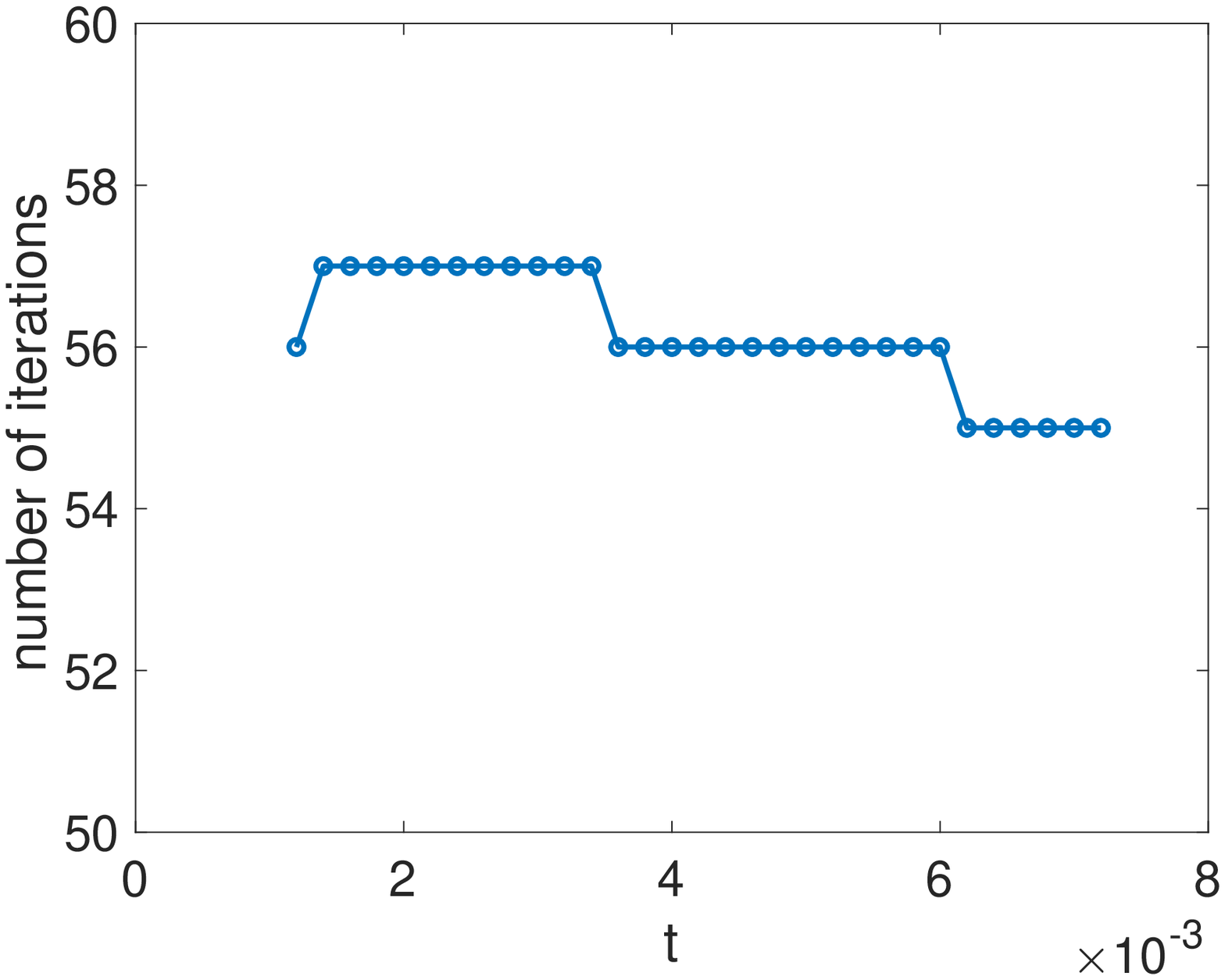}
		\includegraphics[width=0.32\textwidth]{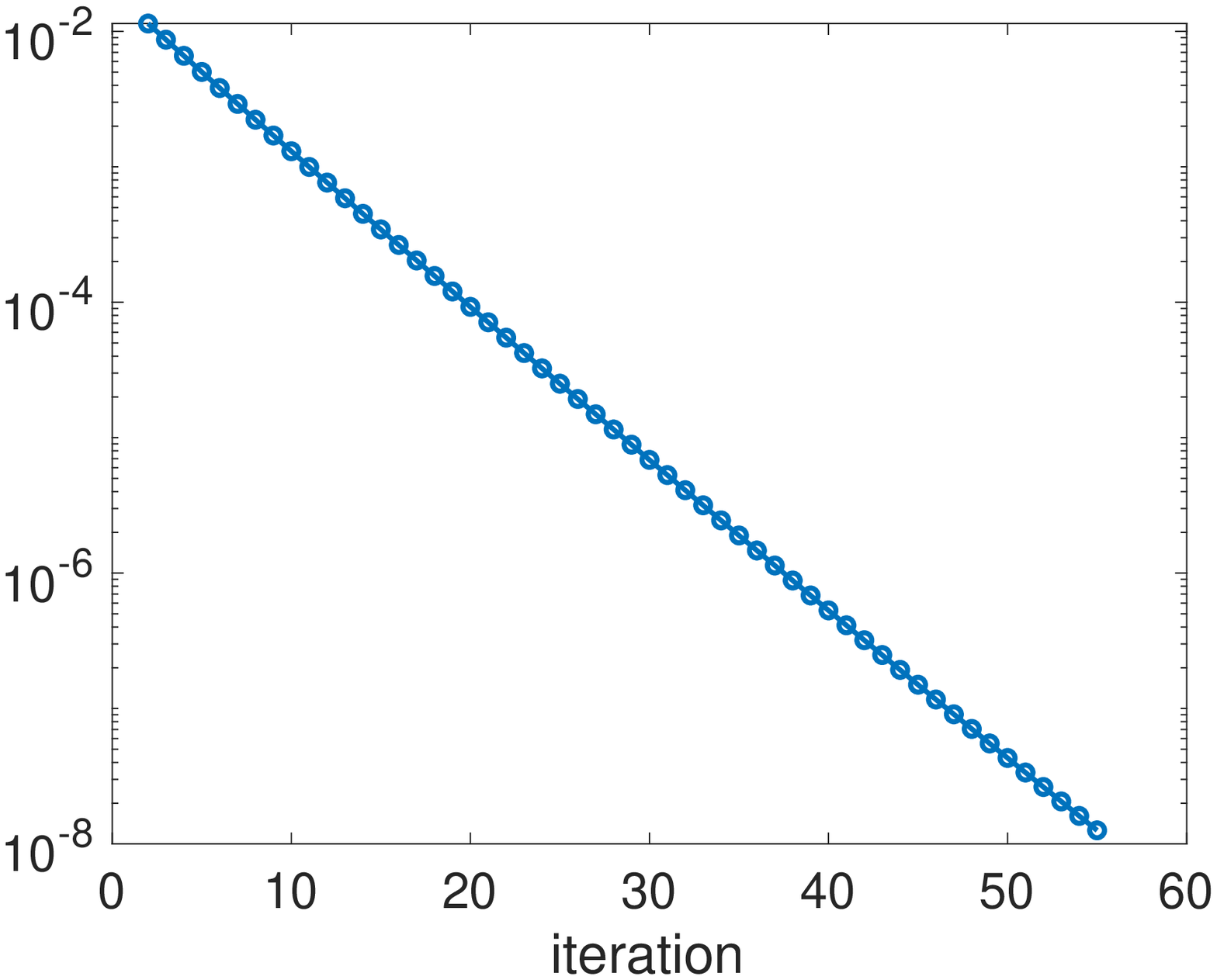}
		\caption{Porous medium equation with $m=2$ and computational domain $[-1,1]$. Left: evolution of $\rho$ compared with exact solution. Middle: number of iterations needed in Algorithm 1 within each outer time step. The tolerance in \eqref{tol} is chosen to be Tol=$10^{-8}$. Right: iteration error of Algorithm 1 in the first step of outer variational scheme. Numerical parameters are $\Delta x = 0.04$, $\tau = 2\times 10^{-4}$, $\eps = 0.005 $ in \eqref{W-phi} and $\eta = 0.2$ in Algorithm 1. }
		\label{fig:porous-MD}
	\end{figure}

	\subsubsection{Aggregation equation}
	Next we consider a nonlocal aggregation equation of the form
	\begin{align} \label{aggeqn1}
	\partial_t \rho = \nabla\cdot (\rho \nabla W*\rho)  \,, \quad W(x) = \frac{|x|^2}{2} - \text{ln}(|x|),
	\end{align}
	where the interaction kernel $W$ is repulsive at short length scales and attractive at longer distances. This equation admits a unique equilibrium profile 
	\begin{equation} \label{equi-agg}
	\rho_\infty (x) =   \frac{1}{\pi} \sqrt{(2-x^2)_+} \,.
	\end{equation}
	In practice, to avoid evaluation of $W(x)$ at $x=0$, we set $W(0)$ to equal the average value of $W$ on the cell of width $2h$ centered at 0, i.e., $W(0) = \frac{1}{2h} \int_{-h}^{h} W(x) \rd x$, where we compute this value analytically. (See also \cite{CCWW18} for a similar treatment.)
	In Fig.~\ref{fig:agg}, we compute \eqref{aggeqn1} with initial data
	\[
	\rho(x,0) = \frac{1}{\sqrt{2\pi}\sigma} e^{-\frac{x^2}{2\sigma^2}} + 10^{-8}\,.
	\]
	The left picture displays the evolution of $\rho$. At $t=3$, the solution has reached the steady state, which matches the analytical formula represented by the dashed curve. The right plot shows the exponential decay of the energy, where the red dashed line indicates the decay rate. We also explore the convergence of our algorithm in Fig.~\ref{fig:agg2}. As seen in the upper left picture, the number of iterations needed in reaching the tolerance has shown some heterogeneity with respect to the outer time. More specifically, at a few times, such as $t=0.528$, a significantly larger number of iterations is needed. More detailed plots on how the relative error \eqref{tol} evolves are displayed in the upper right and lower left figures, in which a few representative plots of the error are given. At $t=0.528$, which corresponding to the first peak, we also plot the solution $\rho$ at this time and the previous time (i.e., $t=0.512$), with a zoom-in plot near the left propagating front of the solution. It is shown that, at the location $x=-1.24$, there is a sharp transition in the solution. That is, $\rho_{10}$ goes from $5.5996\times 10^{-7}$ to $2.1726\times 10^{-4}$, which results in around 389 increase in magnitude. Similar increase are observed at time corresponding to the rest two peaks. So we believe that the deterioration in the convergence is due to such a rapid transition in the solution. 
	
	\begin{figure}[!ht]
		\includegraphics[width=0.45\textwidth]{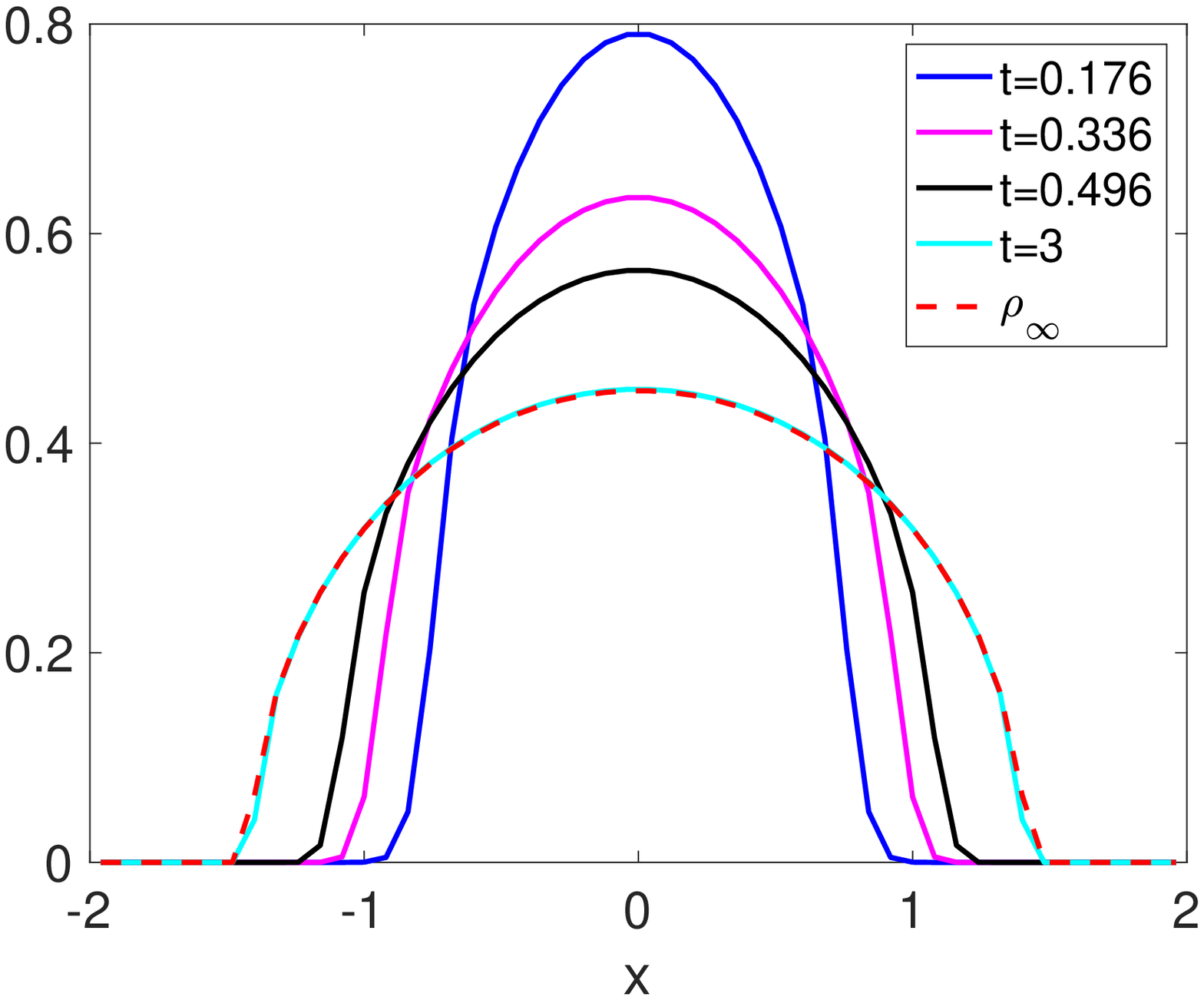}
		\includegraphics[width=0.45\textwidth]{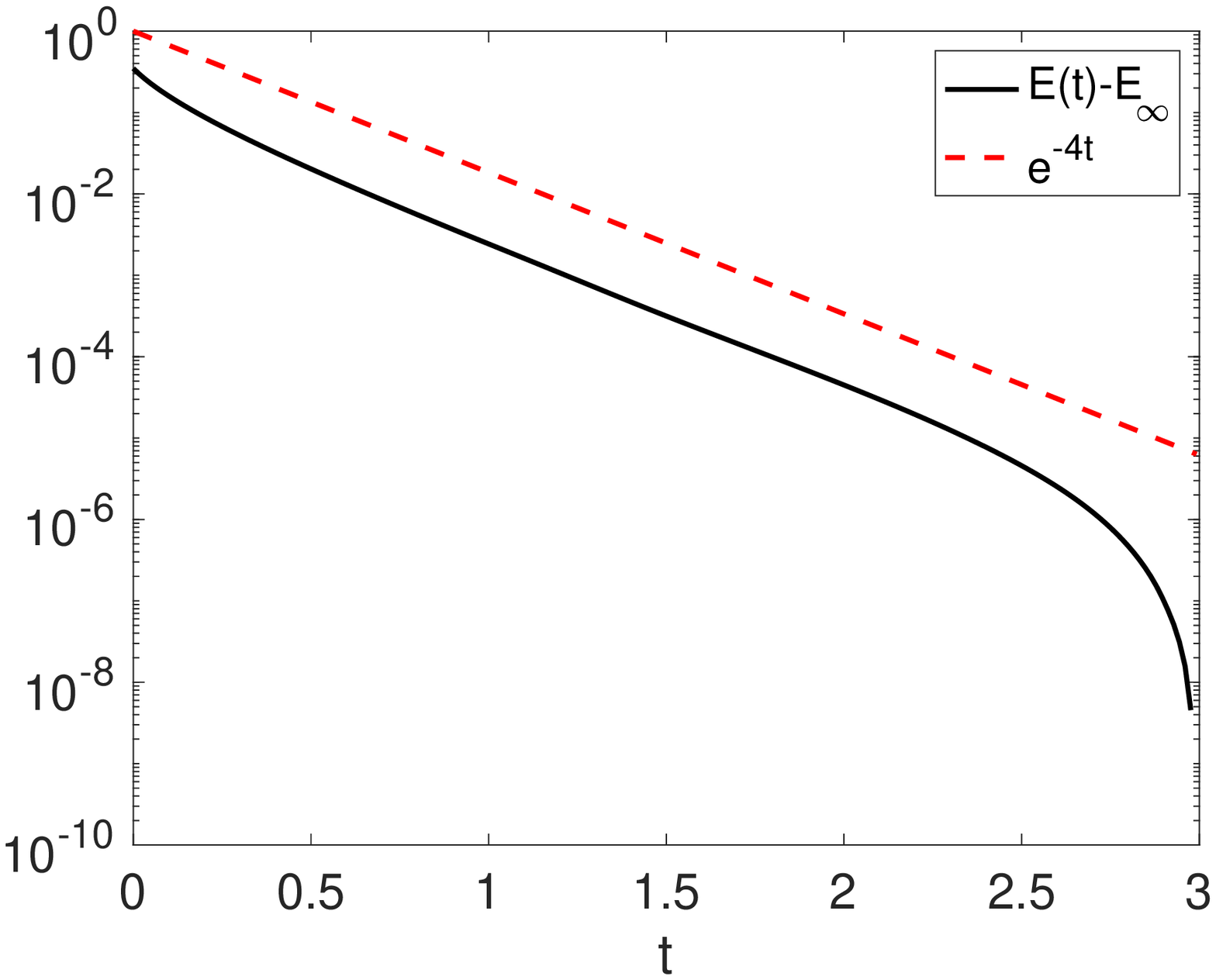}
		\caption{Aggregation equation with computational domain $[-2,2]$ and initial data . Left: evolution of $\rho$, $\rho_\infty$ is given by the analytical formula \eqref{equi-agg}. Right: exponential decay of energy.  Numerical parameters are $\Delta x = 0.08$, $\tau = 0.016$, $\eps = 0.1$ in \eqref{W-phi} and $\eta = 0.8$ in Algorithm 1. }
		\label{fig:agg}
	\end{figure}
	
	\begin{figure}[!ht]
	    \includegraphics[width=0.45\textwidth]{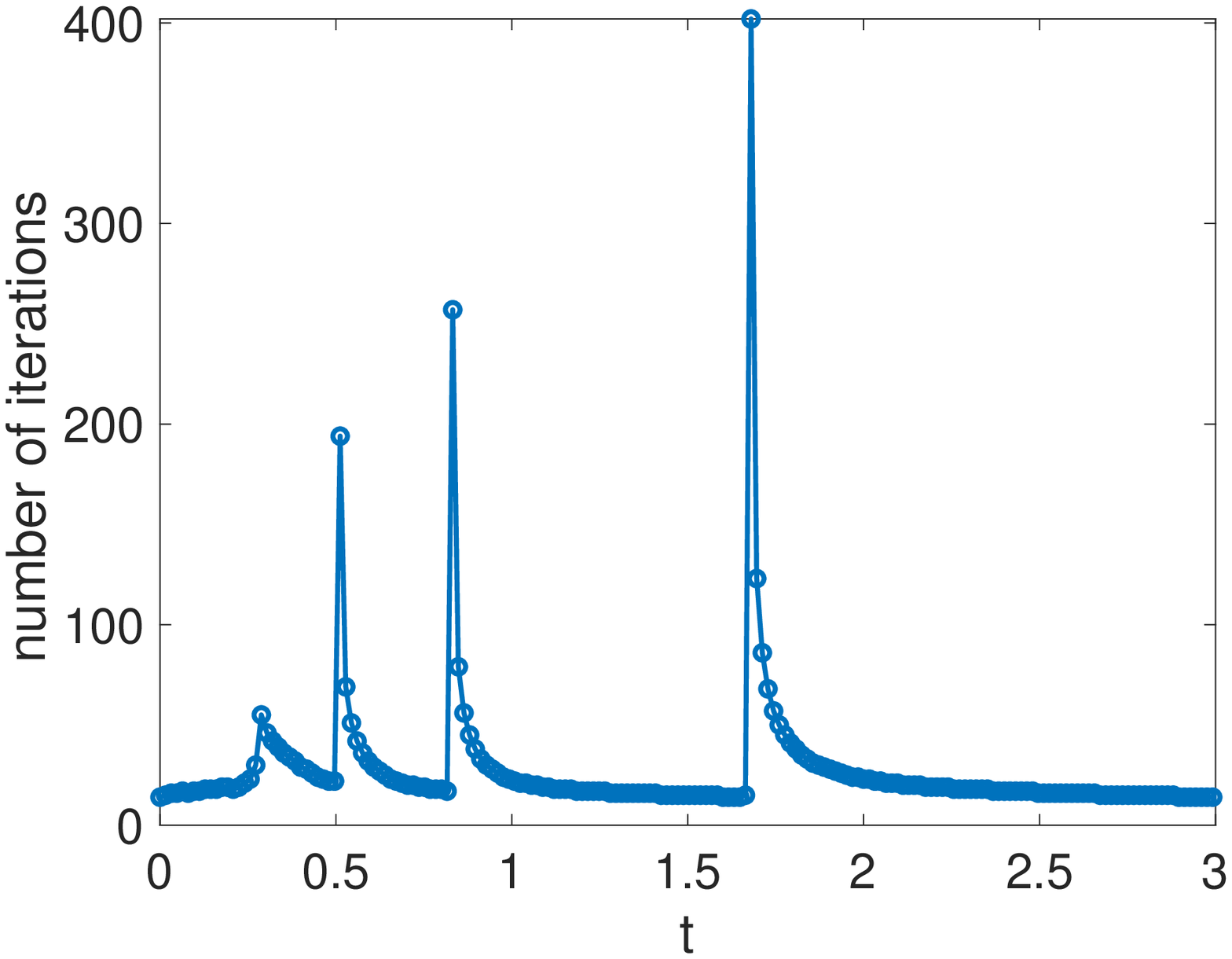}
		\includegraphics[width=0.45\textwidth]{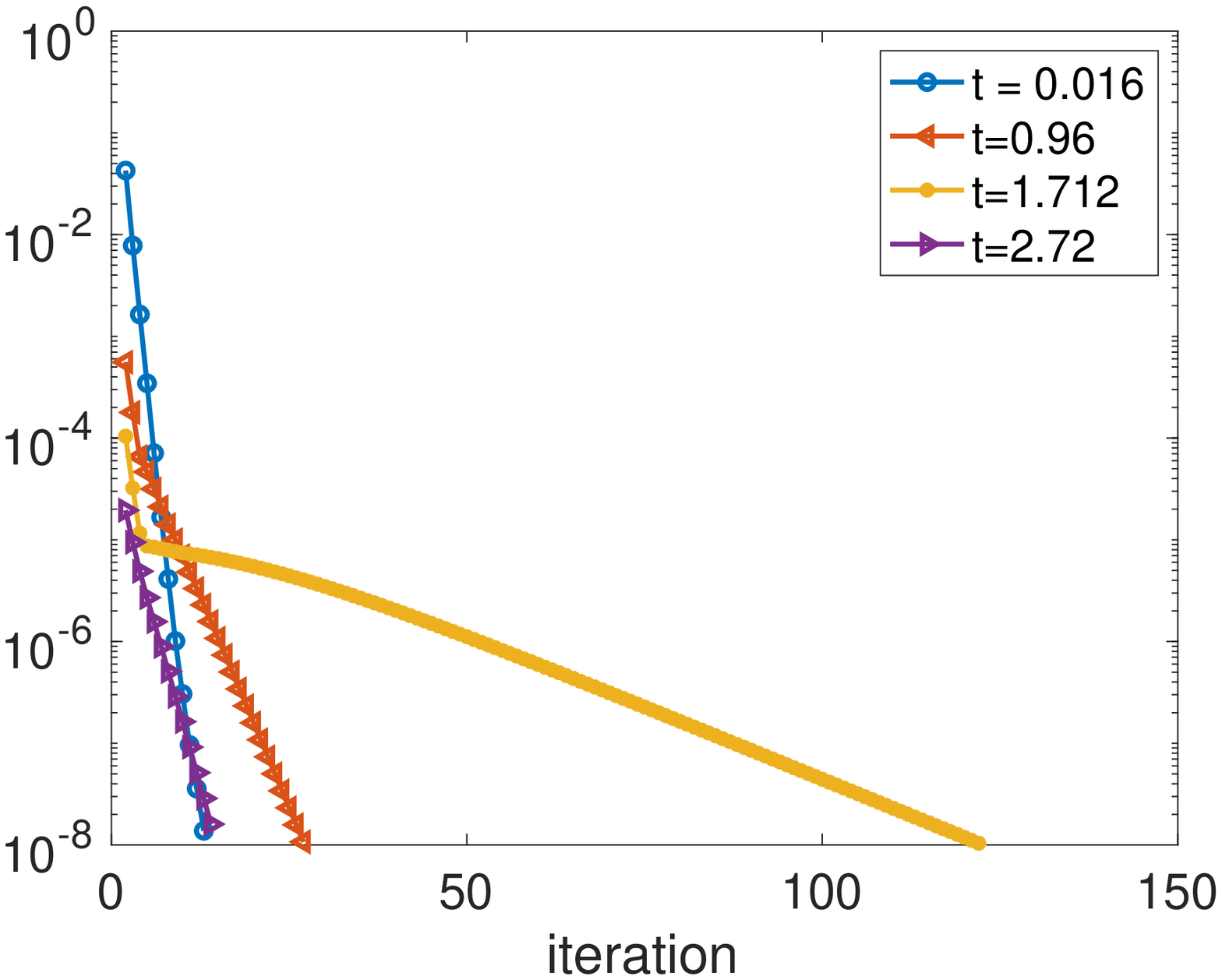}
	    \\ 
		\includegraphics[width=0.45\textwidth]{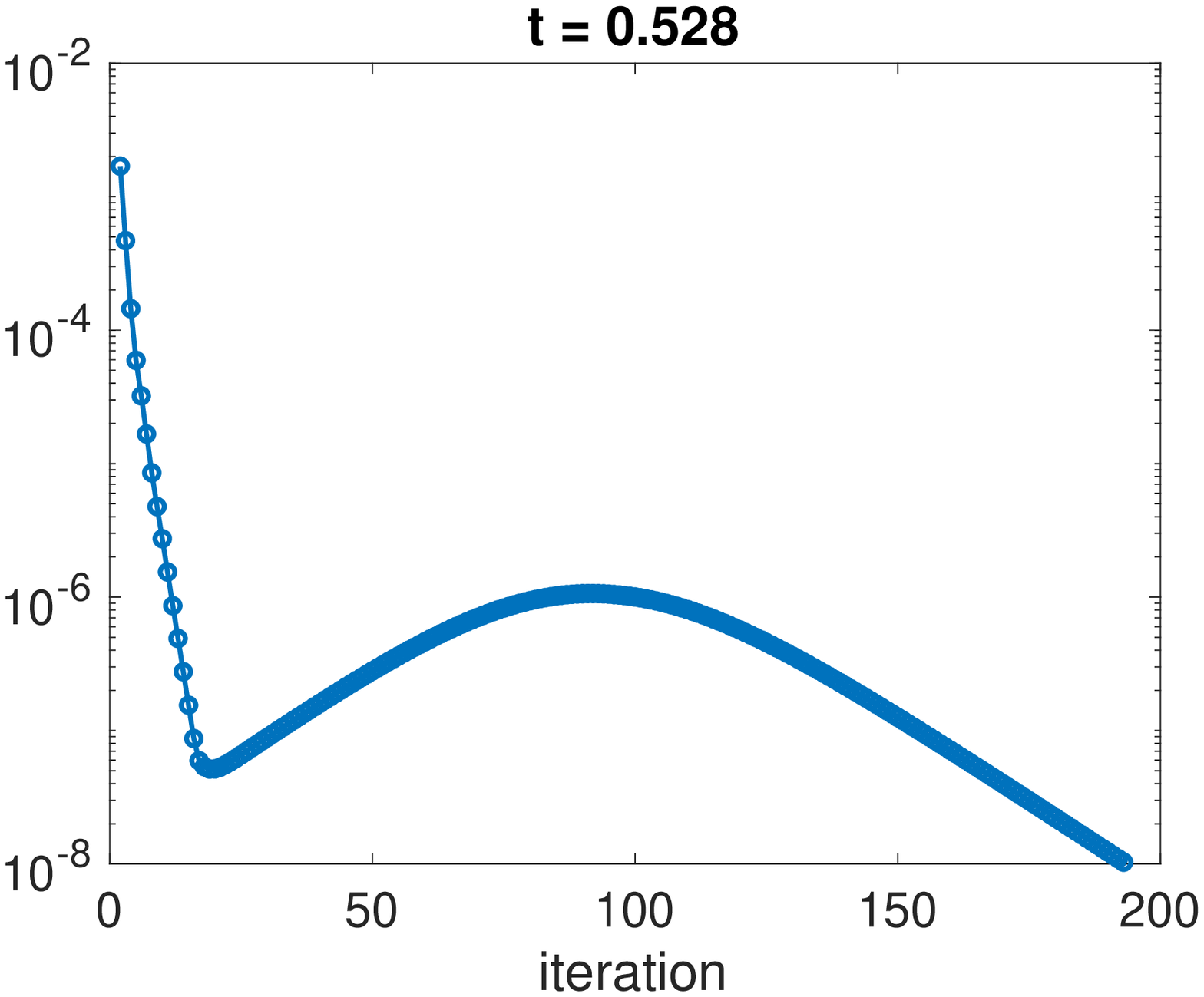}
		\includegraphics[width=0.45\textwidth]{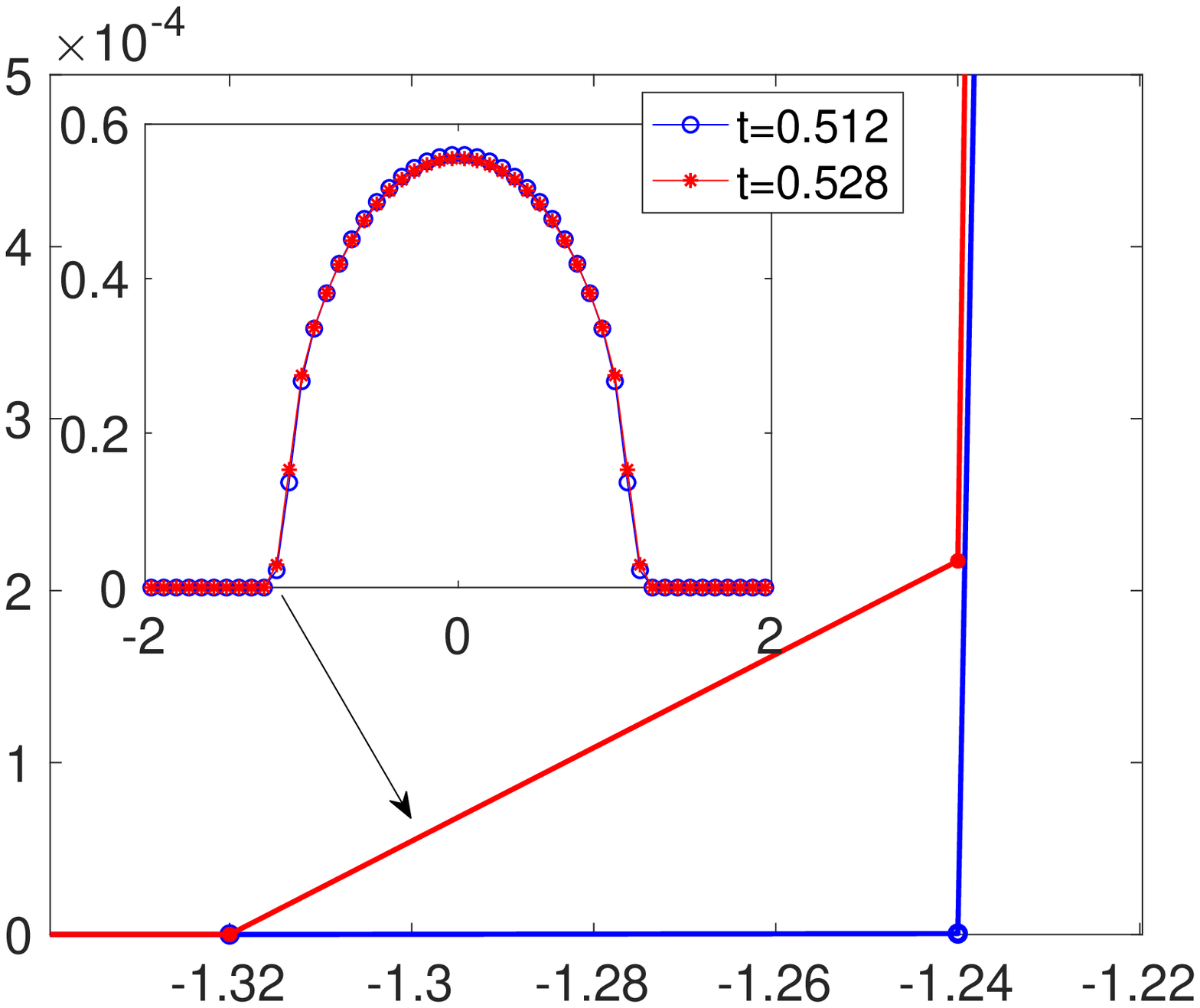}
		\caption{Convergence in the aggregation equation. Upper left: number of iterations needed within each outer time step. Upper right: several typical plot of the convergence (e.g., relative error versus iteration) at different outer times. Lower left: a hard to converge scenario at $t=0.528$. Lower right: plot of the corresponding solution at two consecutive times $t=0.512$ and $t=0.528$.} 
		\label{fig:agg2}
	\end{figure}
	
	In comparison, we also considered the variable metric algorithm \eqref{qNewton1}. With the choice of $\Phi(\rho) = \frac{1}{2\tau} \|  \rho - \rho_n \|_{\Hnrhoi}$, and the algorithm takes the same form as \eqref{variM0}, but with $\energy(\rho) = \half\int \rho(W\ast \rho)  \rd x$. The evolution of $\rho$ is given in Fig.~\ref{fig:agg3}. Here we choose a smaller iteration step $\eta = 0.01$, but the positivity of the solution can still not be preserved, and oscillation round zero values of $\rho$ is generated and amplified along time (compare $\rho$ at $t=1.616$ with $t=3$). 
		\begin{figure}[!ht]
		\centering
	    \includegraphics[width=0.5\textwidth]{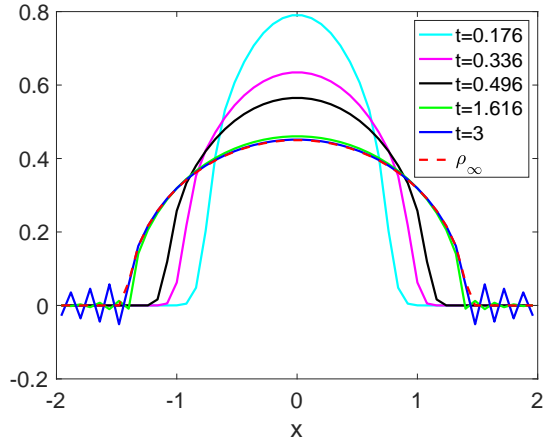}
	    \caption{Computation of aggregation equation using variable metric algorithm \eqref{variM0} with $\eta = 0.01$, $\Delta x = 0.08$, $\tau = 0.016$. }
	    \label{fig:agg3}
	    \end{figure}
	
\subsection{Cahn-Hillard equation with degenerate mobility}	
	Cahn-Hillard equation has first been introduced to study phase separation in binary alloys, and later extended to many other fields such as image inpainting \cite{BEG06} and math biology \cite{GLNS18}. To put it on the same foot as \eqref{gd}, we write it in the gradient flow form:
	\begin{equation} \label{CH}
	   \partial_t u = \nabla \cdot \left(M(u) \nabla \variEu \right)\,,
	\end{equation}
	where $u$ represents the difference in the local concentration of two components in the alloy, and $M(u) = 1-u^2 \geq 0$ is a diffusional mobility. $\energy$ is the energy functional 
	\begin{equation} \label{energy-CH}
	\energy(u) = \int_\Omega (\frac{\alpha^2}{2} |\nabla u|^2 + \Psi(u) ) \rd x \,,
	\end{equation}
	where the first term penalizes large gradients and models the capillary effects. The second term is the homogeneous free energy. A typical form is the Ginzburg-Landau potential $\Psi(u) = \frac{1}{4} (1-u^2)^2$ or logarithmic potential $\Psi(u) = \frac{\theta}{2} \left[ (1+u) \log \left(\frac{1+u}{2} \right) + (1-u) \log\left(\frac{1-u}{2} \right) \right] + \frac{\theta_c}{2}(1-u^2)$ for $u\in (-1,1)$, where $\theta<\theta_c$ are two positive constants. 
	
	As before, we solve \eqref{CH} using the minimizing movement scheme. More precisely, we obtain $u_\np$ by solving 
	\begin{equation} \label{JKO-CH}
	u_0 = u_{\textrm{in}}, \quad u_{n+1} = \arg\min_{u \in \UU} \left\{  \frac{1}{2\tau }\|u - u_n\|_{\Hnui}^2 + \energy(u) \right\} \,,
	\end{equation}
	where $\UU = \{ \int_\Omega u(x) \rd x = \int_\Omega u_\textrm{in}(x) \rd x, -1\leq u \leq 1\}$, and $\|f \|_{\Hnui}^2  = \int_{\RR^d} f(x) \Hnui f(x) \rd x$. Here $\Hnu$ is the negative weighted Laplacian $\Hnu = -\nabla \cdot (M(u_n) \nabla )$, and $\Hnui$ is the pseudo-inverse of $\Hnu$.

	\subsubsection{Mirror descent algorithm}
	It has been proven that $u$ will stay within the interval $(-1,1)$ due either to the singularity in the free energy or degeneracy of the mobility \cite{EG96, AW07}.
    In order to maintain such a bound, we choose $\Phi$ in \eqref{MD1} to be:
	\begin{equation*} \label{CH-phi}
	\Phi(u) = \frac{1}{2\tau }\|u - u_n\|_{\Hnui}^2 + \eps_1 \int (u+1) \log (u+1) \rd x + \eps_2 \int (1-u) \log (1-u) \rd x \,,
	\end{equation*}
	then the mirror descent becomes
	\[
	\frac{\delta \Phi}{\delta u}(u^\kp)- \frac{\delta \Phi}{\delta u}(u^k) = -\eta \left[ \frac{1}{\tau} \Hnui(u^k - u_n) + \variEu (u^k) \right] \,,
	\]
	which simplifies to
	\begin{align} 
	& u^\kp + \tau \Hnu [\eps_1\log (1+u^\kp) + \eps_2 \log (1-u^\kp)] \nonumber
	\\ & \hspace{1cm}= u^k +  \tau \Hnu[ \eps_1 \log (1+u^k) - \eps_2 \log(1-u^k)] 
     - \eta \left[ u^k - u_n + \tau \Hnu \variEu (u^k) \right]\,. \label{MD-CH2}
	\end{align}
	The discretization of $\Hnu$ is the same as in \eqref{disD} except that one replace $\rho_n$ by $M(u_n)$. In solving \eqref{MD-CH2} for $u^{k+1}$, Newton's method will be used and a similar preconditioner as in \eqref{eqn0523} wil be employed. We omit the details as they are very similar to Section~\ref{sec:W-MD}. 
	
	\subsubsection{An example}
	Here we consider a one dimensional example in \cite{BBG99}. Choose $\Psi(u) = \half (1-u^2)$, $\alpha = 0.1$ in \eqref{energy-CH} and let initial condition be 
	\[
	u_\textrm{in}(x) = \left\{ \begin{array}{cc} \cos\left( \frac{x-\half}{\alpha}\right) -1\,, & \textrm{if~} |x-\half| \leq \frac{\pi \alpha }{2} \\ -1\,, & \textrm{other} \end{array} \right. \,.
	\]
	Then the steady state takes the form
	\begin{equation} \label{CH-steady}
	u_\infty(x) = \left\{ \begin{array}{cc} \frac{1}{\pi} \left[ 1 + \cos\left( \frac{\pi - \half}{\alpha} \right) \right] -1 & \textrm{if}~ |x-\half| \leq \pi \alpha \\  -1 & \textrm{other} \end{array} \right. \,.
	\end{equation}
	The results are collected in Fig.~\ref{fig:CH}. The upper figures show the evolution of the density $\rho$ and decay of the energy. The lower left figure displays the number of iterations within each outer time steps, and the spikes again correspond to the rapid transition of the solution near $-1$. All three figures are obtained via the mirror descent algorithm. On the other hand, we implemented the variable metric algorithm with $\Phi(u) = \frac{1}{2\tau} \frac{1}{2\tau }\|u - u_n\|_{\Hnui}^2 $ and the profile of $u$ is given in the lower right plot of Fig.~\ref{fig:CH}. Here with a much smaller choice of iteration step, i.e., $\eta = 0.00025$ as compared to $\eta = 0.02$ in mirror descent, the lower bound of $u$ is still violated, and results in a wrong steady state. 
	\begin{figure}
	    \centering
	    \includegraphics[width=0.45\textwidth]{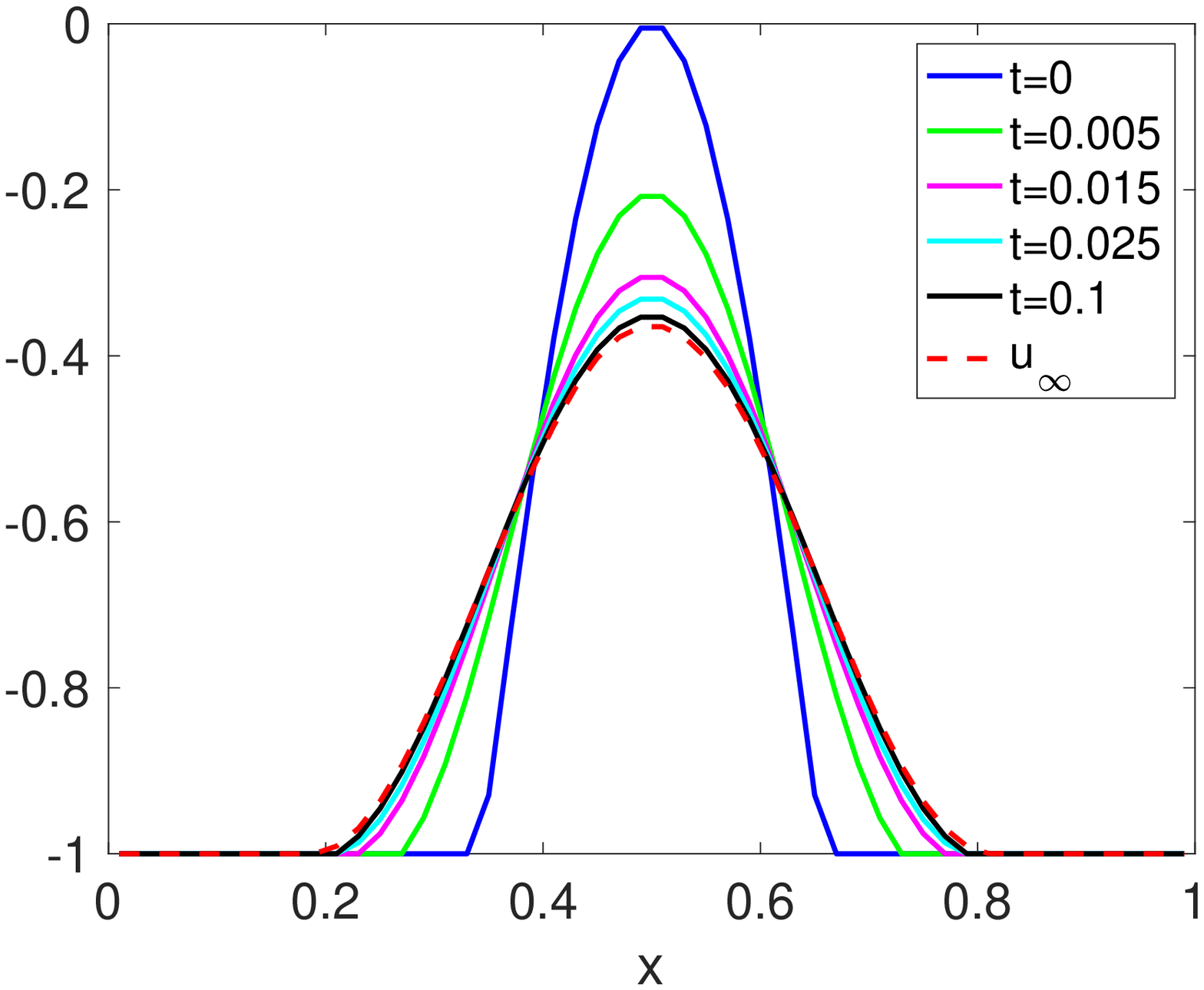}
	    \includegraphics[width=0.45\textwidth]{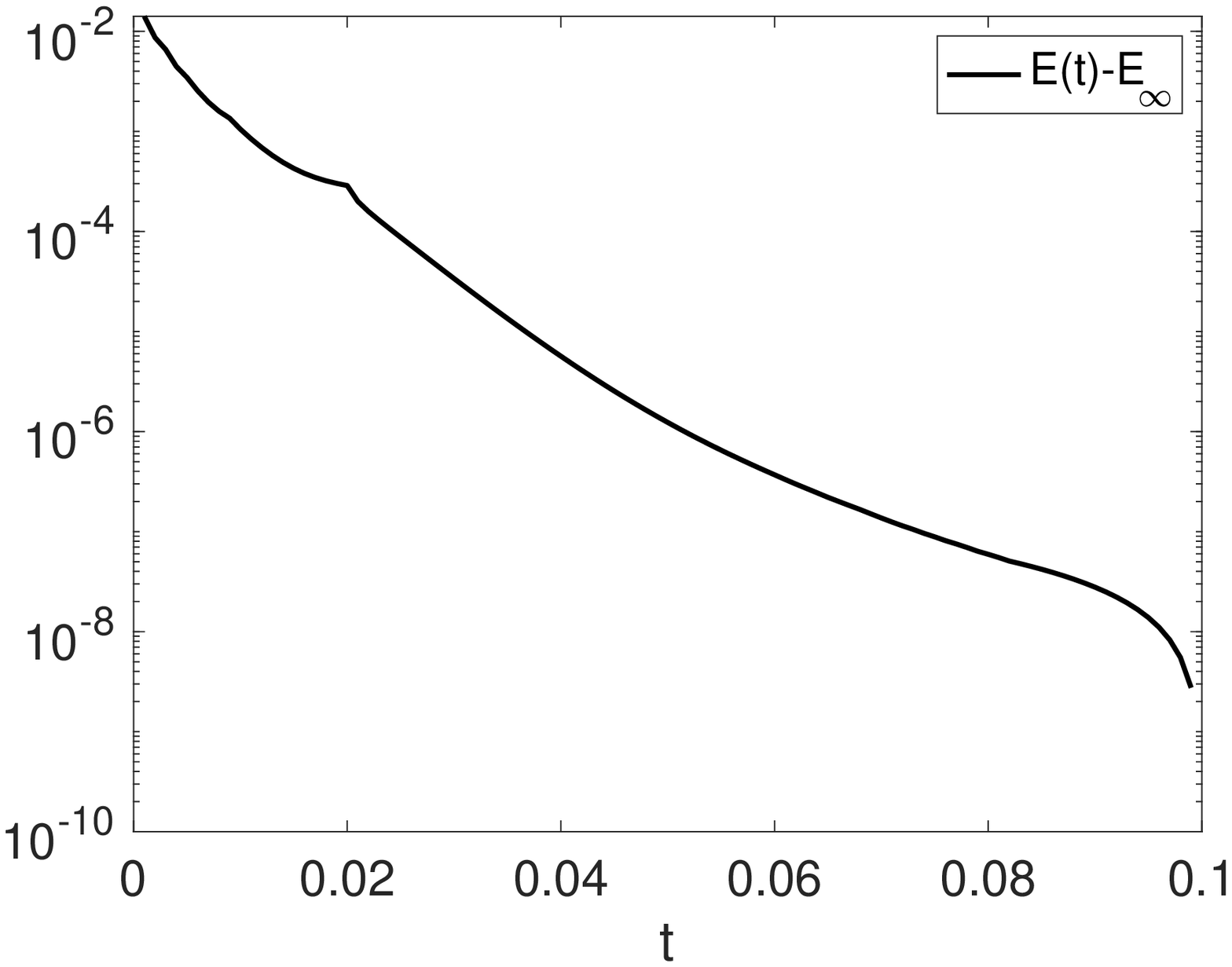} \\
	    \includegraphics[width=0.45\textwidth]{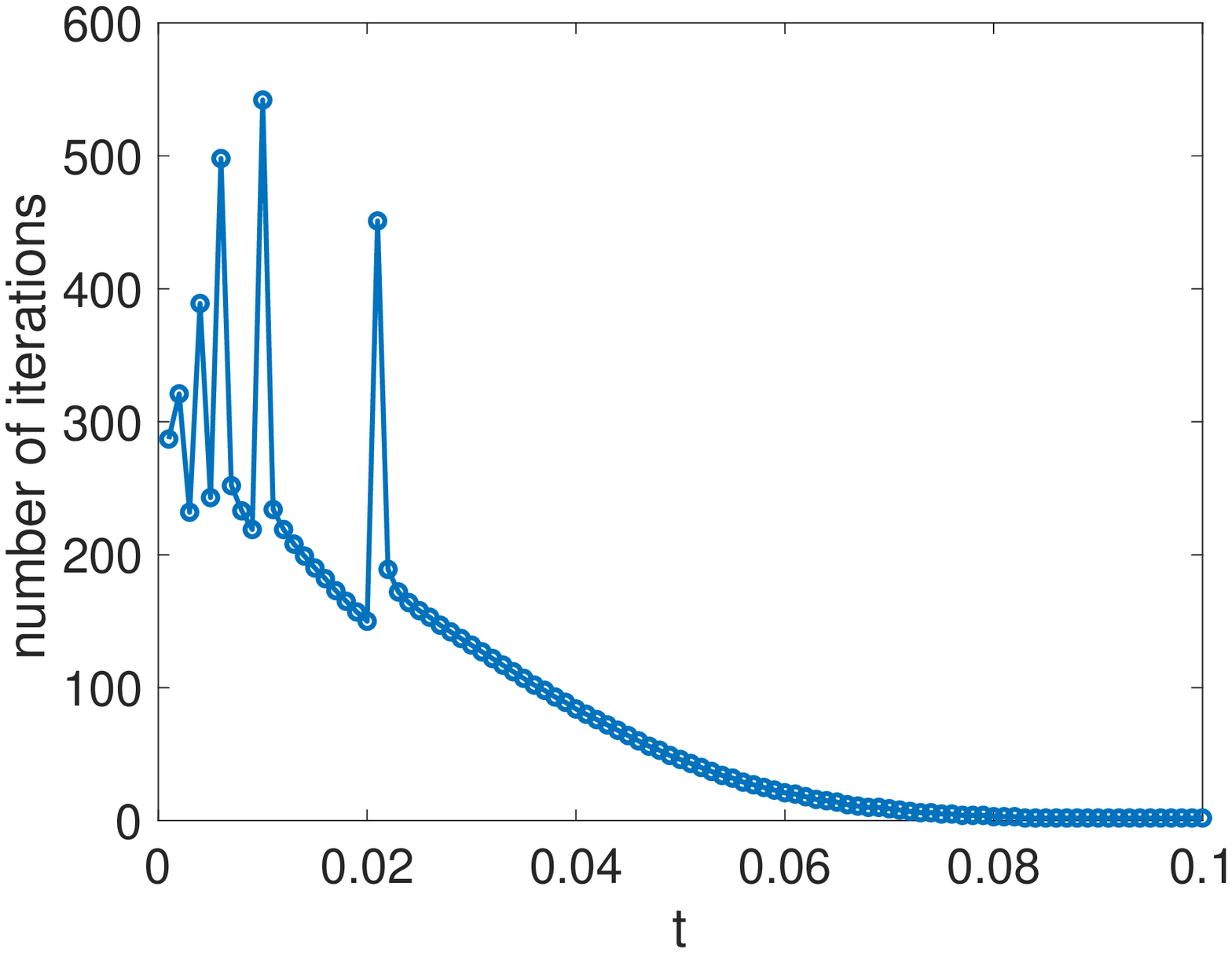}
	    \includegraphics[width=0.45\textwidth]{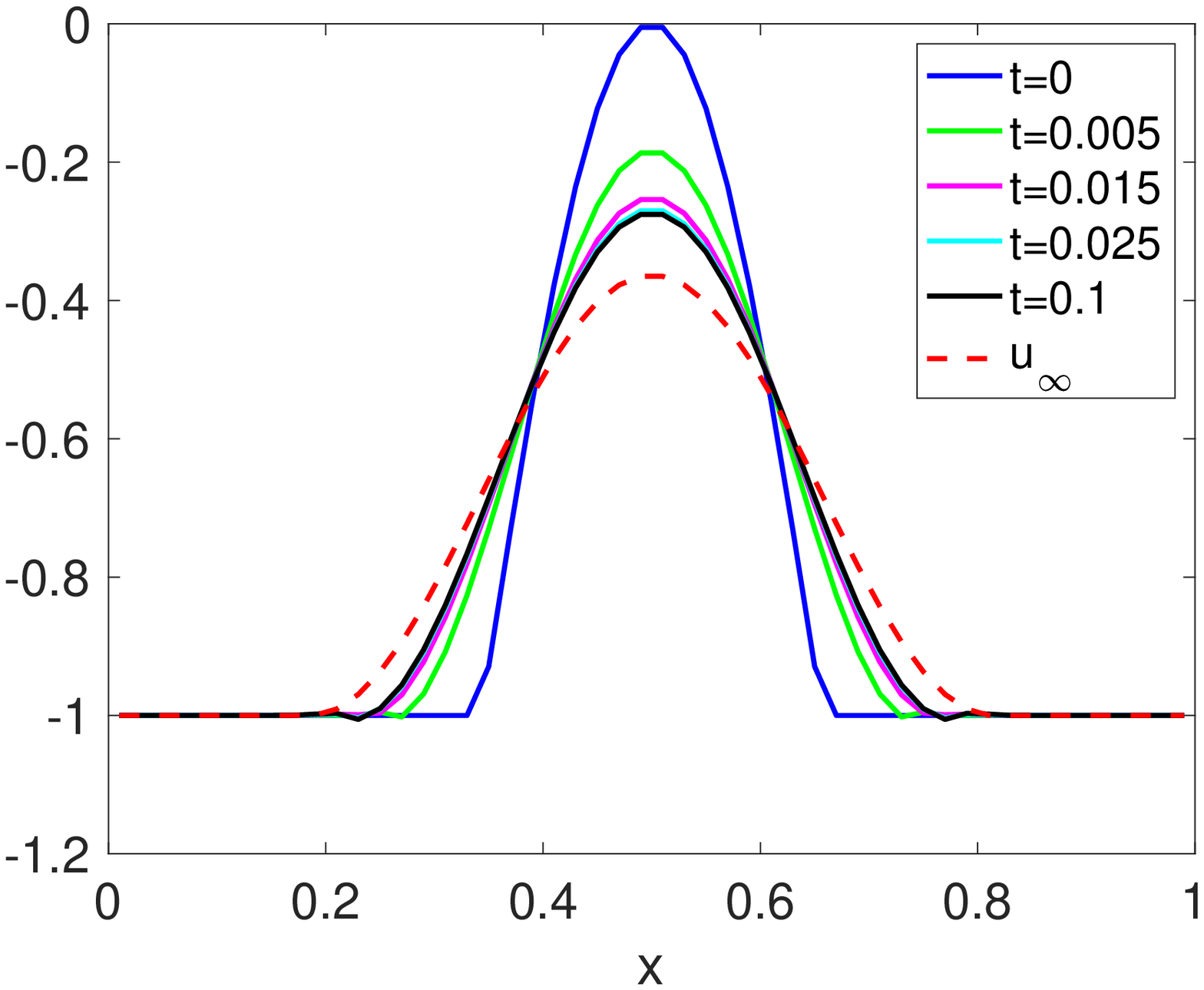}
	    \caption{Cahn-Hillard equation with degenerate mobility. Upper left: evolution of $u$ using mirror descent algorithm. The red dashed curve is given by \eqref{CH-steady}. Upper right: exponential decay of the relative energy. Lower left: number of iterations needed within each outer time step. Lower right: evolution of $u$ using variable metric algorithm. 
	    Numerical parameters are $\Delta x = 0.02$, $\tau = 10^{-3}$, $\eps = 0.5$, $\eta = 0.02$ (mirror descent), and $\eta = 0.00025$ (variable metric). }
	    \label{fig:CH}
	\end{figure}

\section{Conclusion}
	In this paper, we consider a mirror descent algorithm, where the metric is induced by a convex function, whose Hessian is an approximation of the Hessian of the objective function. The advantage of this algorithm is two-fold. On one hand, the mirror descent framework gives a natural way to incorporate the bound constraint of the solution. On the other hand, the Hessian information used in building the metric leads to improved rate of convergence. To put such an advantage on a rigorous footing, we first formulate a gradient flow of the algorithm, in which the constraints are incorporated as a to-be-determined vector. Form this formulation, we can draw connection between the mirror descent and more general variable metric algorithms. Then the improved rate of convergence is proved following the two stage approach in Newton type methods. In return, the proof we obtained for the mirror descent can lend itself to quasi-Newton methods to show the global convergence. We also apply the algorithm to two cases, the Wasserstein gradient flow and Cahn-Hillard equation with degenerate mobility, and demonstrate its effectiveness.   
	
	\bibliography{MD_Hessian_ref}
	\bibliographystyle{siam}
	
\end{document}